\documentclass[twoside,11pt,a4paper]{article}
\usepackage[T1]{fontenc}
\usepackage{fancyhdr}
\usepackage{float}
\usepackage{mathrsfs}
\usepackage{amsmath}
\usepackage{amsthm}
\usepackage{amssymb}
\usepackage{graphicx}
\usepackage {hyperref}
\numberwithin{equation}{section}
\numberwithin{figure}{section}
\numberwithin{table}{section}
\theoremstyle{remark}
\newtheorem{notation}{\protect\notationname}
\theoremstyle{plain}
\newtheorem{thm}{\protect\theoremname}[section]
\theoremstyle{plain}
\newtheorem{prop}{\protect\propositionname}[section]

\date{}
\newcommand{\Auteur}{\scriptsize K.T. S. SOBAH and A. S. d'ALMEIDA}
\newcommand{\TitreArticle}{\scriptsize EXISTENCE AND UNIQUENESS OF SOLUTION FOR AN INITIAL-BOUNDARY VALUE PROBLEM}

\newcommand{\R}{\mathbb{R}}
\newcommand{\p}{\ensuremath{\partial}}
\usepackage{graphics,graphicx}

\providecommand{\notationname}{Notation}
\providecommand{\propositionname}{Proposition}

\providecommand{\theoremname}{Theorem}

\begin{document}
\title{ON THE EXISTENCE AND UNIQUENESS OF CLASSICAL SOLUTION FOR AN INITIAL-BOUNDARY
VALUE PROBLEM FOR A DISCRETE BOLTZMANN SYSTEM IN TWO SPACE DIMENSIONS}
\author{Koudzo Togbévi Selom SOBAH\textsuperscript{\textsuperscript{1}}\textsuperscript{}
and Amah Séna D'ALMEIDA\textsuperscript{\textsuperscript{2}}}

\maketitle

\textsuperscript{1,2}Department of Mathematics, Faculty of Sciences
and Laboratory of Mathematics and Applications, University of Lomé,
Lomé, TOGO

\begin{abstract}
The initial-boundary value problem for the two-dimensional regular
four-velocity discrete boltzmann system is analyzed in a rectangle.
The existence and uniqueness of a classical global positive solution,
bounded with its first partial derivatives are proved for a range
of bounded data by the use of fixed points tools. A bound for the
solution and its partial derivatives is provided. 
\end{abstract}
\textbf{Key words and phrases:} discrete velocity models, initial-boundary
value, existence, uniqueness, fixed point theorems.

\textbf{2020 Mathematics Subject Classification:} 76A02, 76M28
\section*{Author's Note}	
	This is the accepted version of the article published in \textit{Advances in Mathematics: Scientific Journal}, Vol.~14, No.~1 (2025), pp.~73--102.\\
	The final published version is available at: \url{https://doi.org/10.37418/amsj.14.1.5} 

\section{Introduction }

Discrete velocity models of gas are simplified models of the Boltzmann
equation obtained by assuming that the velocities of the gas particles
belong to a finite set of vectors. The nonlinear integrodifferential
Boltzmann equation is replaced by a system of semilinear hyperbolic
equations associated to the number densities of the particles having
the given velocities. After the pionner works of Broadwell \cite{3,4}
who introduced the first physically convenient models in the sense
that they can model actual gas flows and the theory for the general
discrete velocity model for binary collision given by Gatignol \cite{1},
the discrete kinetic theory of gas develops in two directions: the
mathematical study of the kinetic equations encompassing the existence
and the uniqueness theory as well as the construction of exact solutions
and the modelling and the resolution of flow problems.

The existence theory for the discrete velocity models of the Boltzmann
equation, not only supports the mathematical understanding of these
models but also underpins the reliability of numerical methods used
in engineering and physics. The good mathematical structure of the
kinetic equations associated to these models lead to the rapid development
of the mathematical theory of discrete velocity models. Many results
concerning the proof of the global existence and the uniqueness of
the solutions of the initial-boundary value problem have been obtained
in the one-dimensional case \cite{5,6,11,14,16}. Most of these studies
concern the so-called three velocity and four velocity Broadwell models
which are the symmetrical models obtained from the six and the eight
velocity spatial models of Broadwell by a symmetry with respect to
any coordinate plane and axis respectively. Exact solutions have been
proposed for the three velocity Broadwell model\cite{7}.\\

The situation is quite different for multi-dimensional problems even
in the steady case. In \cite{2}, using techniques based on the fractional
steps method, the problem of existence and uniqueness of the solution
of the initial boundary value problem is solved for the two velocity
Carleman model. In the steady case, the boundary value problem for
the general two-dimensional four velocity Broadwell model is investigated
in \cite{10,defoou,d almeida,Nicou=00003D0000E9} the existence of
a solution is proved and exact solutions are built. An extension to
a fifteen velocity three speed discrete model is done in \cite{d almeida}.

In this work, the initial-boundary value problem for a two-dimensional
four velocity model of Broadwell is considered in a rectangle; we
prove for a range of bounded initial and boundary data, the existence
and uniqueness of the classical global positive solution which is
bounded with its first partial derivatives and we provide a bound
for the solution and its partial derivatives.

The paper is organized as follows. In section \ref{presentationequations}
we briefly describe the model, state the initial-boundary value problem
and present the main result of the paper which is proved in section
\ref{sec:Resolution}. In section \ref{sec:Positivite} we establish
the positivity of the solution of the initial-boundary value problem.

\section{Statement of the problem\label{presentationequations}}

\subsection{The discrete velocity model}

The general plane four velocity discrete models of Broadwell denoted
by $B_{\theta}$, $\theta\in\left[0,\dfrac{\pi}{2}\right[$ are among
the simplest discrete velocity models and have been used to study
initial and boundary value problems in one dimension \cite{4,9,defoou}
and to build exact solutions \cite{10,Nicou=00003D0000E9}. In the
basis $(\vec{e_{1}},\vec{e_{2}})$ of orthonormal reference $(O,\vec{e_{1}},\vec{e_{2}})$
of the plane ${\R}^{2}$ its velocities are $\vec{u_{1}}=c(cos\theta,sin\theta)$,
$\vec{u_{2}}=c(-sin\theta,cos\theta)$, $\vec{u_{3}}=-\vec{u_{2}}$,
$\vec{u_{4}}=-\vec{u_{1}}$, where $\theta=angle(\vec{e_{1}},\vec{u_{1}})$
accounts of the orientation of the discrete velocity model with respect
to the reference.

Let $N_{i}(t,x,y)$ be the number density of the gas molecules with
velocity $\vec{u}_{i},\ i=1,2,3,,4$ at the time $t$ and at the position
$M(x,y)$ the kinetic equations of the model are:

\begin{equation}
\left\{ \begin{array}{lcl}
\medskip\dfrac{\p N_{1}}{\p t}+ccos\theta\dfrac{\p N_{1}}{\p x}+csin\theta\dfrac{\p N_{1}}{\p y} & = & Q\\
\medskip\dfrac{\p N_{2}}{\p t}-csin\theta\dfrac{\p N_{2}}{\p x}+ccos\theta\dfrac{\p N_{2}}{\p y} & = & -Q\\
\medskip\dfrac{\p N_{3}}{\p t}+csin\theta\dfrac{\p N_{3}}{\p x}-ccos\theta\dfrac{\p N_{3}}{\p y} & = & -Q\\
\medskip\dfrac{\p N_{4}}{\p t}-ccos\theta\dfrac{\p N_{4}}{\p x}-csin\theta\dfrac{\p N_{4}}{\p y} & = & Q.
\end{array}\right.\label{b2}
\end{equation}

\[
Q=2cS\left(N_{2}N_{3}-N_{1}N_{4}\right).
\]

\begin{figure}[H]
\centering{}\includegraphics[scale=0.4]{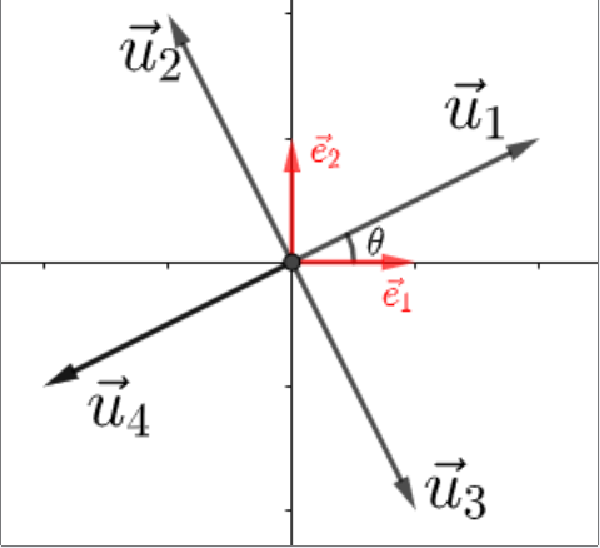} \caption{The model $B_{\theta}$}
\label{modelfigure1}
\end{figure}

The total density $\rho$ and the macroscopic velocity $\overrightarrow{U}(U,V)$
of a gas described by the model are defined by: 
\begin{eqnarray}
\begin{array}{l}
\rho=N_{1}+N_{2}+N_{3}+N_{4}\\
\rho U=\cos(\theta)[N_{1}-N_{4}]-\sin(\theta)[N_{2}-N_{3}]\\
\rho V=\sin(\theta)[N_{1}-N_{4}]+\cos(\theta)[N_{2}-N_{3}]\\
\end{array}\label{macroscopicvariables}
\end{eqnarray}
The Maxwellian densities of the model associated with the macroscopic
variables $\rho$, $U$ and $V$ are given by the relations: 
\begin{eqnarray}
\begin{array}{l}
N_{1M}=\dfrac{\rho}{4}\left[1+\cos(2\theta)\left(U^{2}-V^{2}\right)+2UV\sin(2\theta)+2U\cos(\theta)+2V\sin(\theta)\right]\\
\\
N_{2M}=\dfrac{\rho}{4}\left[1-\cos(2\theta)\left(U^{2}-V^{2}\right)-2UV\sin(2\theta)+2V\cos(\theta)-2U\sin(\theta)\right]\\
\\
N_{3M}=\dfrac{\rho}{4}\left[1-\cos(2\theta)\left(U^{2}-V^{2}\right)-2UV\sin(2\theta)-2V\cos(\theta)+2U\sin(\theta)\right]\\
\\
N_{4M}=\dfrac{\rho}{4}\left[1+\cos(2\theta)\left(U^{2}-V^{2}\right)+2uv\sin(2\theta)-2U\cos(\theta)-2V\sin(\theta)\right]
\end{array}\label{maxwellian}
\end{eqnarray}
The mild problem in consideration in the sequel results from the modelling
of a gas flow in a rectangular box by the model $B_{0}$.

\subsection{Initial-boundary value problem }

Given $\Omega=\left[a_{1},b_{1}\right]\times\left[a_{2},b_{2}\right]\subset\R^{2}$,
$I=\left[0;T\right]\subset\R$ , we set $\mathscr{P}=\left[0;T\right]\times\left[a_{1};b_{1}\right]\times\left[a_{2};b_{2}\right]$
and consider the system $\Sigma^{0}$ defined by: 
\begin{align}
{\displaystyle {\scriptstyle {\textstyle \dfrac{\partial N_{1}}{\partial t}+c\dfrac{\partial N_{1}}{\partial x}=}}} & {\scriptstyle {\textstyle Q(N),\;\left(t,x,y\right)\in\mathring{\mathscr{P}}}}\label{eq:erter}\\
\dfrac{\partial N_{2}}{\partial t}+c\dfrac{\partial N_{2}}{\partial y}= & -Q(N),\;\left(t,x,y\right)\in\mathring{\mathscr{P}}\\
\dfrac{\partial N_{3}}{\partial t}-c\dfrac{\partial N_{3}}{\partial y}= & -Q(N),\;\left(t,x,y\right)\in\mathring{\mathscr{P}}\\
\dfrac{\partial N_{4}}{\partial t}-c\dfrac{\partial N_{4}}{\partial x}= & Q(N),\;\left(t,x,y\right)\in\mathring{\mathscr{P}}\label{eq:zzzfra}\\
N_{i}\left(0,x,y\right)= & N_{i}^{0}\left(x,y\right),\;\left(x,y\right)\in\left[a_{1};b_{1}\right]\times\left[a_{2};b_{2}\right],i=1,\cdots,4\label{eq:lsos}\\
N_{1}\left(t,a_{1},y\right)= & N_{1}^{-}\left(t,y\right),\;\left(t,y\right)\in\left[0;T\right]\times\left[a_{2};b_{2}\right]\label{eq:losso}\\
N_{2}\left(t,x,a_{2}\right)= & N_{2}^{-}\left(t,x\right),\;\left(t,x\right)\in\left[0;T\right]\times\left[a_{1};b_{1}\right]\label{eq:mppos}\\
N_{3}\left(t,x,b_{2}\right)= & N_{3}^{+}\left(t,x\right),\;\left(t,x\right)\in\left[0;T\right]\times\left[a_{1};b_{1}\right]\label{eq:mpsss}\\
N_{4}\left(t,b_{1},y\right)= & N_{4}^{+}\left(t,y\right),\;\left(t,y\right)\in\left[0;T\right]\times\left[a_{2};b_{2}\right]\label{eq:lsoo}
\end{align}
\begin{equation}
N_{1}^{0}\left(a_{1},y\right)=N_{1}^{-}\left(0,y\right),\;y\in\left[a_{2};b_{2}\right]\label{eq:kqiiq}
\end{equation}
\begin{equation}
N_{2}^{0}\left(x,a_{2}\right)=N_{2}^{-}\left(0,x\right),\;x\in\left[a_{1};b_{1}\right]\label{eq:qkkqo}
\end{equation}
\begin{equation}
N_{3}^{0}\left(x,b_{2}\right)=N_{3}^{+}\left(0,x\right),\;x\in\left[a_{1};b_{1}\right]\label{eq:ssqz}
\end{equation}
\begin{equation}
N_{4}^{0}\left(b_{1},y\right)=N_{4}^{+}\left(0,y\right),\;y\in\left[a_{2};b_{2}\right]\label{eq:jqiiq}
\end{equation}

where 
\begin{equation}
Q\left(N\right)=2cS\left(N_{2}N_{3}-N_{1}N_{4}\right),\label{eq:solls}
\end{equation}
$N_{i}^{0},i=1,\cdots,4,$ are the initial data and $N_{1}^{-},N_{2}^{-},N_{3}^{+},N_{4}^{+}$
the boundary data.We assume in the sequel that $N_{i}^{0},\left(i=1,2,3,4\right),N_{1}^{-},N_{2}^{-},N_{3}^{+},N_{4}^{+}$
are non-negative and continuous, that they have bounded and continuous
first order partial derivatives. In the sequel, $C\left(X,Y\right)$
denotes the set of continuous functions from the set $X$ into the
set $Y.$ Our aim is to prove for the system $\Sigma^{0}$, the existence
of non-negative solutions in $C\left(\begin{array}{r}
\mathscr{P}\end{array};\R^{4}\right)$ (thus bounded solutions) that have bounded first order partial derivatives.

\subsection{Main theorem}
\begin{notation}
For every function $u:X\longrightarrow\R$ whose domain is $D_{u}\subset\begin{array}{r}
X\end{array}$ such that $u$ is bounded on $D_{u},$ let us denote $\left\Vert u\right\Vert _{\infty}=\sup_{x\in D_{u}}\left|u\left(x\right)\right|$
and for $U=\left(U_{i}\right)_{i=1}^{4}:X\longrightarrow\R^{4},$
such that every $U_{i}:D_{u}\left(\subset\begin{array}{r}
X\end{array}\right)\longrightarrow\R$ is bounded, $\left\Vert U\right\Vert =\max_{1\leq i\leq4}\left\Vert U_{i}\right\Vert _{\infty}.$ 
\end{notation}
\begin{notation}
For $u:D_{u}\longrightarrow\R$ such that $u\equiv u\left(\alpha,\beta\right)$
is bounded on $D_{u}$ and such that $\dfrac{\partial u}{\partial\alpha},\dfrac{\partial u}{\partial\beta}$
are bounded on their domain, let us set

$\left\Vert u\right\Vert _{1}=\max\left\{ \left\Vert u\right\Vert _{\infty},\left\Vert \dfrac{\partial u}{\partial\alpha}\right\Vert _{\infty},\left\Vert \dfrac{\partial u}{\partial\beta}\right\Vert _{\infty}\right\} .$ 
\end{notation}
Let us consider the following parameters

\[
p\equiv4cS\left(1+2\cdot\max\Biggl\{4T;{\displaystyle \frac{2}{c}}\left(b_{1}-a_{1}\right);{\displaystyle \frac{1}{c}}\left(b_{2}-a_{2}\right)\Biggr\}\right);
\]
\begin{multline*}
q\equiv\\
\max\Biggl\{\max\left\{ 1;2c\right\} \left\Vert N_{1}^{0}\right\Vert _{1};\left(1+c\right)\left\Vert N_{1}^{-}\right\Vert _{1};\max\left\{ 1;2c\right\} \left\Vert N_{2}^{0}\right\Vert _{1};\max\left\{ 2;\left(1+c\right)\right\} \left\Vert N_{2}^{-}\right\Vert _{1};\\
\max\left\{ 1;2c\right\} \left\Vert N_{3}^{0}\right\Vert _{1};\max\left\{ 2;\left(1+c\right)\right\} \left\Vert N_{3}^{+}\right\Vert _{1};\max\left\{ 1;2c\right\} \left\Vert N_{4}^{0}\right\Vert _{1};\left(2+c\right)\left\Vert N_{4}^{+}\right\Vert _{1}\Biggr\}.
\end{multline*}

We prove in the sequel the following result 
\begin{thm}
\label{thm:Suppose-.-Then}Suppose $pq\leq\dfrac{1}{4}$. Then the
system $\Sigma^{0}$ \ref{eq:erter}-\ref{eq:lsoo} has an unique
non-negative solution 
\[
N=\left(N_{1},N_{2},N_{3},N_{4}\right)\in C\left(\left[0;T\right]\times\left[a_{1},b_{1}\right]\times\left[a_{2},b_{2}\right];\R^{4}\right)
\]
such that

\begin{align}
{\displaystyle \left\Vert N\right\Vert } & \leq\dfrac{1+\sqrt{1-4pq}}{2p},
\end{align}

$\dfrac{\partial N_{i}}{\partial t},\dfrac{\partial N_{i}}{\partial x},\dfrac{\partial N_{i}}{\partial y}$
are defined in $\left]0;T\right[\times\left]a_{1},b_{1}\right[\times\left]a_{2},b_{2}\right[$
except possibly on a finite number of planes including the four planes
with respective equations 
\begin{equation}
-ct+x=a_{1};-ct+y=a_{2};ct+y=b_{2};ct+x=b_{1};\label{eq:ldoopz}
\end{equation}

$\dfrac{\partial N_{i}}{\partial t},\dfrac{\partial N_{i}}{\partial x},\dfrac{\partial N_{i}}{\partial y}$
are continuous and bounded, for $i=1,2,3,4,$ and satisfy

\begin{multline}
{\displaystyle \max_{1\leq i\leq4}\left\{ \left\Vert N_{i}\right\Vert _{\infty},\left\Vert \dfrac{\partial N_{i}}{\partial t}\right\Vert _{\infty},\left\Vert \dfrac{\partial N_{i}}{\partial x}\right\Vert _{\infty},\left\Vert \dfrac{\partial N_{i}}{\partial y}\right\Vert _{\infty}\right\} }\leq\\
\max\left\{ 1,\dfrac{2}{c}\right\} \dfrac{1+\sqrt{1-4pq}}{2p}.\label{eq:loosqz}
\end{multline}
\end{thm}

\section{Non-negativity of the solution\label{sec:Positivite}}

Let's consider the change of variables $\mathscr{F}:\left(t,x,y\right)\longmapsto\left(\eta_{1},\eta_{2},\eta_{3}\right)$
such that $\eta_{1}=x/c,$$\eta_{2}=t/2-x/2c+y/2c$ and $\eta_{3}=t/2-x/2c-y/2c.$
We have $\dfrac{D\left(\eta_{1},\eta_{2},\eta_{3}\right)}{D\left(t,x,y\right)}=\dfrac{1}{2c^{2}}\neq0$
and the inverse of $\mathscr{F}$ is defined by $\mathscr{F}^{-1}:\left(\eta_{1},\eta_{2},\eta_{3}\right)\longmapsto\left(t,x,y\right)$such
that $t=\eta_{1}+\eta_{2}+\eta_{3},$ $x=c\eta_{1}$ and $y=c\eta_{2}-c\eta_{3}.$The
transformed of the mixed problem $\Sigma^{0},$ through the change
of variables ,is the following problem $\Sigma^{1}:$ 
\begin{align}
\dfrac{\partial\widetilde{N_{1}}}{\partial\eta_{1}} & =Q(\widetilde{N}),\:\left(\eta_{1},\eta_{2},\eta_{3}\right)\in\mathring{\overbrace{\mathscr{F}\left(\mathscr{P}\right)}}\equiv\mathring{\mathscr{P}'}\label{eq:ssdffz}\\
\dfrac{\partial\widetilde{N_{2}}}{\partial\eta_{2}} & =-Q(\widetilde{N}),\:\left(\eta_{1},\eta_{2},\eta_{3}\right)\in\mathring{\mathscr{P}'}\\
\dfrac{\partial\widetilde{N_{3}}}{\partial\eta_{3}} & =-Q(\widetilde{N}),\:\left(\eta_{1},\eta_{2},\eta_{3}\right)\in\mathring{\mathscr{P}'}\\
-\dfrac{\partial\widetilde{N_{4}}}{\partial\eta_{1}}+\dfrac{\partial\widetilde{N_{4}}}{\partial\eta_{2}}+\dfrac{\partial\widetilde{N_{4}}}{\partial\eta_{3}} & =Q(\widetilde{N}),\:\left(\eta_{1},\eta_{2},\eta_{3}\right)\in\mathring{\mathscr{P}'}\label{eq:qqqqsa}
\end{align}
\begin{equation}
\widetilde{N_{1}}\left(-\eta_{2}-\eta_{3},\eta_{2},\eta_{3}\right)=N_{1}^{0}\left(-c\eta_{2}-c\eta_{3},c\eta_{2}-c\eta_{3}\right)\label{eq:dlz}
\end{equation}
\begin{equation}
\widetilde{N_{2}}\left(\eta_{1},-\eta_{1}-\eta_{3},\eta_{3}\right)=N_{2}^{0}\left(c\eta_{1},-c\eta_{1}-2c\eta_{3}\right)\label{eq:smqpo}
\end{equation}
\begin{equation}
\widetilde{N_{3}}\left(\eta_{1},\eta_{2},-\eta_{1}-\eta_{2}\right)=N_{3}^{0}\left(c\eta_{1},c\eta_{1}+2c\eta_{2}\right)\label{eq:ksio}
\end{equation}
\begin{equation}
\widetilde{N_{4}}\left(-\eta_{2}-\eta_{3},\eta_{2},\eta_{3}\right)=N_{4}^{0}\left(-c\eta_{2}-c\eta_{3},c\eta_{2}-c\eta_{3}\right)\label{eq:uiiq}
\end{equation}
\begin{equation}
\widetilde{N_{1}}\left({\displaystyle \frac{1}{c}}a_{1},\eta_{2},\eta_{3}\right)=N_{1}^{-}\left({\displaystyle \frac{1}{c}}a_{1}+\eta_{2}+\eta_{3},c\eta_{2}-c\eta_{3}\right)\label{eq:lsoo-1}
\end{equation}

\begin{equation}
\widetilde{N_{2}}\left(\eta_{1},\eta_{3}+\dfrac{a_{2}}{c},\eta_{3}\right)=N_{2}^{-}\left(\eta_{1}+2\eta_{3}+\dfrac{a_{2}}{c},c\eta_{1}\right)\label{eq:qmpoe}
\end{equation}
\begin{equation}
\widetilde{N_{3}}\left(\eta_{1},\eta_{2},\eta_{2}-\dfrac{b_{2}}{c}\right)=N_{3}^{+}\left(\eta_{1}+2\eta_{2}-\dfrac{b_{2}}{c},c\eta_{1}\right)\label{eq:oillq}
\end{equation}
\begin{equation}
\widetilde{N_{4}}\left({\displaystyle \frac{1}{c}}b_{1},\eta_{2},\eta_{3}\right)=N_{4}^{+}\left({\displaystyle \frac{1}{c}}b_{1}+\eta_{2}+\eta_{3},c\eta_{2}-c\eta_{3}\right).\label{eq:yui}
\end{equation}

\subsection{Non-negative operator. }
\begin{prop}
Let $\sigma>0.$ The problem (\ref{eq:qqqqsa})-(\ref{eq:yui}) is
equivalent to:
\begin{align}
\dfrac{\partial\widetilde{N_{1}}}{\partial\eta_{1}}+\sigma\rho\left(\widetilde{N}\right)\widetilde{N_{1}} & =\sigma\rho\left(\widetilde{N}\right)\widetilde{N_{1}}+Q\left(\widetilde{N}\right),\:\left(\eta_{1},\eta_{2},\eta_{3}\right)\in\mathring{\overbrace{\mathscr{F}\left(\mathscr{P}\right)}}\equiv\mathring{\mathscr{P}'}\label{eq:ssdffz-1}\\
\dfrac{\partial\widetilde{N_{2}}}{\partial\eta_{2}}+\sigma\rho\left(\widetilde{N}\right)\widetilde{N_{2}} & =\sigma\rho\left(\widetilde{N}\right)\widetilde{N_{2}}-Q\left(\widetilde{N}\right),\:\left(\eta_{1},\eta_{2},\eta_{3}\right)\in\mathring{\mathscr{P}'}\\
\dfrac{\partial\widetilde{N_{3}}}{\partial\eta_{3}}+\sigma\rho\left(\widetilde{N}\right)\widetilde{N_{3}} & =\sigma\rho\left(\widetilde{N}\right)\widetilde{N_{3}}-Q\left(\widetilde{N}\right),\:\left(\eta_{1},\eta_{2},\eta_{3}\right)\in\mathring{\mathscr{P}'}\\
-\dfrac{\partial\widetilde{N_{4}}}{\partial\eta_{1}}+\dfrac{\partial\widetilde{N_{4}}}{\partial\eta_{2}}+\dfrac{\partial\widetilde{N_{4}}}{\partial\eta_{3}}+\sigma\rho\left(\widetilde{N}\right)\widetilde{N_{4}} & =\sigma\rho\left(\widetilde{N}\right)\widetilde{N_{4}}+Q\left(\widetilde{N}\right),\:\left(\eta_{1},\eta_{2},\eta_{3}\right)\in\mathring{\mathscr{P}'}\label{eq:qqqqsa-1}
\end{align}
\end{prop}
with the conditions (\ref{eq:dlz})-(\ref{eq:yui}) and $\rho\left(\widetilde{N}\right)=\sum_{i=1}^{4}\widetilde{N_{i}}.$
\begin{proof}
The proof is obvious so that the equations (\ref{eq:ssdffz-1})-(\ref{eq:qqqqsa-1})
are obtained by adding $\sigma\rho\left(\widetilde{N}\right)\widetilde{N_{i}},$$i=1,\cdots,4$
to the members of equations (\ref{eq:ssdffz})-(\ref{eq:qqqqsa}).
\end{proof}
In the sequel we denote $Q_{i}^{\sigma}\left(\widetilde{N}\right)=\sigma\rho\left(\widetilde{N}\right)\widetilde{N_{i}}+Q\left(\widetilde{N}\right),$$i=1,4$
and $Q_{i}^{\sigma}\left(\widetilde{N}\right)=\sigma\rho\left(\widetilde{N}\right)\widetilde{N_{i}}-Q\left(\widetilde{N}\right),$$i=2,3.$
\begin{prop}
:\label{prop::azeeza-1}

Let $\sigma>0.$ Let $\widetilde{M}=\left(\widetilde{M_{1}},\widetilde{M_{2}},\widetilde{M_{3}},\widetilde{M_{4}}\right)$
be a fixed 4-tuple of continuous functions defined from $\begin{array}{r}
\mathscr{P}'\end{array}$ to $\R$. Let's put $\left|\widetilde{M}\right|=\left(\left|\widetilde{M_{1}}\right|,\left|\widetilde{M_{2}}\right|,\left|\widetilde{M_{3}}\right|,\left|\widetilde{M_{4}}\right|\right)$.
Let's consider the decoupled system $\left(\Sigma_{\sigma,\widetilde{M}}^{1}\right)$
of the following equations : $\widetilde{N}=\left(\widetilde{N_{1}},\widetilde{N_{2}},\widetilde{N_{3}},\widetilde{N_{4}}\right),$
(\ref{eq:ssdffz-1-1})-(\ref{eq:qqqqsa-1-1}) 
\begin{align}
\dfrac{\partial\widetilde{N_{1}}}{\partial\eta_{1}}+\sigma\rho\left(\left|\widetilde{M}\right|\right)\widetilde{N_{1}} & =Q_{1}^{\sigma}\left(\left|\widetilde{M}\right|\right),\:\left(\eta_{1},\eta_{2},\eta_{3}\right)\in\mathring{\overbrace{\mathscr{F}\left(\mathscr{P}\right)}}\equiv\mathring{\mathscr{P}'}\label{eq:ssdffz-1-1}\\
\dfrac{\partial\widetilde{N_{2}}}{\partial\eta_{2}}+\sigma\rho\left(\left|\widetilde{M}\right|\right)\widetilde{N_{2}} & =Q_{2}^{\sigma}\left(\left|\widetilde{M}\right|\right),\:\left(\eta_{1},\eta_{2},\eta_{3}\right)\in\mathring{\mathscr{P}'}\label{eq:skiz}\\
\dfrac{\partial\widetilde{N_{3}}}{\partial\eta_{3}}+\sigma\rho\left(\left|\widetilde{M}\right|\right)\widetilde{N_{3}} & =Q_{3}^{\sigma}\left(\left|\widetilde{M}\right|\right),\:\left(\eta_{1},\eta_{2},\eta_{3}\right)\in\mathring{\mathscr{P}'}\label{eq:losik}\\
-\dfrac{\partial\widetilde{N_{4}}}{\partial\eta_{1}}+\dfrac{\partial\widetilde{N_{4}}}{\partial\eta_{2}}+\dfrac{\partial\widetilde{N_{4}}}{\partial\eta_{3}}+\sigma\rho\left(\left|\widetilde{M}\right|\right)\widetilde{N_{4}} & =Q_{4}^{\sigma}\left(\left|\widetilde{M}\right|\right),\:\left(\eta_{1},\eta_{2},\eta_{3}\right)\in\mathring{\mathscr{P}'}\label{eq:qqqqsa-1-1}
\end{align}
with the conditions (\ref{eq:dlz})-(\ref{eq:yui}). Then for sufficiently
large $\sigma$ , the problem $\left(\Sigma_{\sigma,\widetilde{M}}^{1}\right)$
has an unique continuous and non-negative solution. 
\end{prop}
\begin{proof}
:The problem (\ref{eq:ssdffz-1-1})-(\ref{eq:qqqqsa-1-1}) is a linear
problem associated with (\ref{eq:ssdffz-1})-(\ref{eq:qqqqsa-1}).
Using the conditions (\ref{eq:dlz})-(\ref{eq:yui}), it\textquoteright s
unique solution is:
\begin{multline}
\widetilde{N_{1}}\left(\eta_{1},\eta_{2},\eta_{3}\right)=\\
\left(\int_{-\eta_{2}-\eta_{3}}^{\eta_{1}}e^{\sigma\int_{-\eta_{2}-\eta_{3}}^{s}\rho\left(\left|\widetilde{M}\right|\right)\left(r,\eta_{2},\eta_{3}\right)dr}Q_{1}^{\sigma}\left(\left|\widetilde{M}\right|\right)\left(s,\eta_{2},\eta_{3}\right)ds+\overline{N_{1}^{0}}\left(\eta_{2},\eta_{3}\right)\right)\\
\times e^{-\sigma\int_{-\eta_{2}-\eta_{3}}^{\eta_{1}}\rho\left(\left|\widetilde{M}\right|\right)\left(s,\eta_{2},\eta_{3}\right)ds}\cdot\mathbb{I}_{-c\eta_{2}-c\eta_{3}\geq a_{1}}\left(\eta_{1},\eta_{2},\eta_{3}\right)+\\
\left(\int_{\frac{1}{c}a_{1}}^{\eta_{1}}e^{\sigma\int_{\frac{1}{c}a_{1}}^{s}\rho\left(\left|\widetilde{M}\right|\right)\left(r,\eta_{2},\eta_{3}\right)dr}Q_{1}^{\sigma}\left(\left|\widetilde{M}\right|\right)\left(s,\eta_{2},\eta_{3}\right)ds+\overline{N_{1}^{-}}\left(\eta_{2},\eta_{3}\right)\right)\\
\times e^{-\sigma\int_{\frac{1}{c}a_{1}}^{\eta_{1}}\rho\left(\left|\widetilde{M}\right|\right)\left(s,\eta_{2},\eta_{3}\right)ds}\cdot\mathbb{I}_{-c\eta_{2}-c\eta_{3}\leq a_{1}}\left(\eta_{1},\eta_{2},\eta_{3}\right)\label{aoalal}
\end{multline}
\begin{multline}
\widetilde{N_{2}}\left(\eta_{1},\eta_{2},\eta_{3}\right)=\\
\left(\int_{-\eta_{1}-\eta_{3}}^{\eta_{2}}e^{\sigma\int_{-\eta_{1}-\eta_{3}}^{s}\rho\left(\left|\widetilde{M}\right|\right)\left(\eta_{1},r,\eta_{3}\right)dr}Q_{2}^{\sigma}\left(\left|\widetilde{M}\right|\right)\left(\eta_{1},s,\eta_{3}\right)ds+\overline{N_{2}^{0}}\left(\eta_{1},\eta_{3}\right)\right)\\
\times e^{-\sigma\int_{-\eta_{1}-\eta_{3}}^{\eta_{2}}\rho\left(\left|\widetilde{M}\right|\right)\left(\eta_{1},s,\eta_{3}\right)}\cdot\mathbb{I}_{-c\eta_{1}-2c\eta_{3}\geq a_{2}}\left(\eta_{1},\eta_{2},\eta_{3}\right)+\\
\left({\displaystyle \int_{\eta_{3}+\frac{a_{2}}{c}}^{\eta_{2}}}e^{\sigma\int_{\eta_{3}+\frac{a_{2}}{c}}^{s}\rho\left(\left|\widetilde{M}\right|\right)\left(\eta_{1},r,\eta_{3}\right)dr}Q{}_{2}^{\sigma}\left(\left|\widetilde{M}\right|\right)\left(\eta_{1},s,\eta_{3}\right)ds+\overline{N_{2}^{-}}\left(\eta_{1},\eta_{3}\right)\right)\\
\times e^{-\sigma\int_{\eta_{3}+\frac{a_{2}}{c}}^{\eta_{2}}\rho\left(\left|\widetilde{M}\right|\right)\left(\eta_{1},s,\eta_{3}\right)ds}\cdot\mathbb{I}_{-c\eta_{1}-2c\eta_{3}\leq a_{2}}\left(\eta_{1},\eta_{2},\eta_{3}\right)\label{eq:ppmms}
\end{multline}
\begin{multline}
\widetilde{N_{3}}\left(\eta_{1},\eta_{2},\eta_{3}\right)=\\
\left({\displaystyle \int_{-\eta_{1}-\eta_{2}}^{\eta_{3}}}e^{\sigma\int_{-\eta_{1}-\eta_{2}}^{s}\rho\left(\left|\widetilde{M}\right|\right)\left(\eta_{1},\eta_{2},r\right)dr}Q_{3}^{\sigma}\left(\left|\widetilde{M}\right|\right)\left(\eta_{1},\eta_{2},s\right)ds+\overline{N_{3}^{0}}\left(\eta_{1},\eta_{2}\right)\right)\\
\times e^{-\sigma\int_{-\eta_{1}-\eta_{2}}^{\eta_{3}}\rho\left(\left|\widetilde{M}\right|\right)\left(\eta_{1},\eta_{2},s\right)ds}\cdot\mathbb{I}_{c\eta_{1}+2c\eta_{2}\leq b_{2}}\left(\eta_{1},\eta_{2},\eta_{3}\right)+\\
\left({\displaystyle \int_{\eta_{2}-\frac{b_{2}}{c}}^{\eta_{3}}}e^{\sigma\int_{\eta_{2}-\frac{b_{2}}{c}}^{s}\rho\left(\left|\widetilde{M}\right|\right)\left(\eta_{1},\eta_{2},r\right)dr}Q_{3}^{\sigma}\left(\left|\widetilde{M}\right|\right)\left(\eta_{1},\eta_{2},s\right)ds+\overline{N_{3}^{+}}\left(\eta_{1},\eta_{2}\right)\right)\\
\times e^{-\sigma{\displaystyle \int_{\eta_{2}-\frac{b_{2}}{c}}^{\eta_{3}}}\rho\left(\left|\widetilde{M}\right|\right)\left(\eta_{1},\eta_{2},s\right)ds}\cdot\mathbb{I}_{c\eta_{1}+2c\eta_{2}\geq b_{2}}\left(\eta_{1},\eta_{2},\eta_{3}\right)\label{eq:oslslo}
\end{multline}
\begin{multline}
\widetilde{N_{4}}\left(\eta_{1},\eta_{2},\eta_{3}\right)=\mathbb{I}_{2c\eta_{1}+c\eta_{2}+c\eta_{3}\leq b_{1}}\left(\eta_{1},\eta_{2},\eta_{3}\right)\\
\cdot\left[\int_{0}^{\eta_{1}+\eta_{2}+\eta_{3}}\left(e^{\sigma\int_{0}^{s}\rho\left(\left|\widetilde{M}\right|\right)\left(-r+2\eta_{1}+\eta_{2}+\eta_{3};r-\eta_{1}-\eta_{3};r-\eta_{1}-\eta_{2}\right)dr}\right)\right.\\
\times Q_{4}^{\sigma}\left(\left|\widetilde{M}\right|\right)\left(-s+2\eta_{1}+\eta_{2}+\eta_{3};s-\eta_{1}-\eta_{3};s-\eta_{1}-\eta_{2}\right)ds\\
\left.+\overline{N_{4}^{0}}\left(\eta_{1},\eta_{2},\eta_{3}\right)\right]e^{-\sigma\int_{0}^{\eta_{1}+\eta_{2}+\eta_{3}}\rho\left(\left|\widetilde{M}\right|\right)\left(-s+2\eta_{1}+\eta_{2}+\eta_{3};s-\eta_{1}-\eta_{3};s-\eta_{1}-\eta_{2}\right)ds}\\
+\mathbb{I}_{2c\eta_{1}+c\eta_{2}+c\eta_{3}\geq b_{1}}\left(\eta_{1},\eta_{2},\eta_{3}\right)\\
\cdot\left[\int_{0}^{\left(-\eta_{1}+\frac{1}{c}b_{1}\right)}\left(e^{\sigma\int_{0}^{s}\rho\left(\left|\widetilde{M}\right|\right)\left(-r+\frac{1}{c}b_{1};r+\eta_{1}+\eta_{2}-\frac{1}{c}b_{1};r+\eta_{1}+\eta_{3}-\frac{1}{c}b_{1}\right)dr}\right)\right.\\
\times Q_{4}^{\sigma}\left(\left|\widetilde{M}\right|\right)\left(-s+\frac{1}{c}b_{1},s+\eta_{1}+\eta_{2}-\dfrac{1}{c}b_{1},s+\eta_{1}+\eta_{3}-\dfrac{1}{c}b_{1}\right)ds\\
\left.+\overline{N_{4}^{+}}\left(\eta_{1},\eta_{2},\eta_{3}\right)\right]e^{-\sigma\int_{0}^{\left(-\eta_{1}+\frac{1}{c}b_{1}\right)}\rho\left(\left|\widetilde{M}\right|\right)\left(-s+\frac{1}{c}b_{1};s+\eta_{1}+\eta_{2}-\frac{1}{c}b_{1};s+\eta_{1}+\eta_{3}-\frac{1}{c}b_{1}\right)ds}\label{eq:ayyeiz-1}
\end{multline}
where 
\begin{equation}
\begin{cases}
\overline{N_{1}^{0}}\left(\eta_{2},\eta_{3}\right)\equiv N_{1}^{0}\left(-c\eta_{2}-c\eta_{3},c\eta_{2}-c\eta_{3}\right)\\
\text{and }\overline{N_{1}^{-}}\left(\eta_{2},\eta_{3}\right)\equiv N_{1}^{-}\left({\displaystyle \frac{1}{c}}a_{1}+\eta_{2}+\eta_{3},c\eta_{2}-c\eta_{3}\right)
\end{cases}\label{eq:iiskzi}
\end{equation}

\begin{equation}
\begin{cases}
\overline{N_{2}^{0}}\left(\eta_{1},\eta_{3}\right)\equiv N_{2}^{0}\left(c\eta_{1},-c\eta_{1}-2c\eta_{3}\right)\\
\text{ and }\overline{N_{2}^{-}}\left(\eta_{1},\eta_{3}\right)\equiv N_{2}^{-}\left(\eta_{1}+2\eta_{3}+\dfrac{a_{2}}{c},c\eta_{1}\right)
\end{cases}\label{eq:ikoik}
\end{equation}
\begin{equation}
\begin{cases}
\overline{N_{3}^{0}}\left(\eta_{1},\eta_{2}\right)\equiv N_{3}^{0}\left(c\eta_{1},c\eta_{1}+2c\eta_{2}\right)\\
\text{ and }\overline{N_{3}^{+}}\left(\eta_{1},\eta_{2}\right)\equiv N_{3}^{+}\left(\eta_{1}+2\eta_{2}-\dfrac{b_{2}}{c},c\eta_{1}\right)
\end{cases}\label{eq:olooki}
\end{equation}
\begin{equation}
\begin{cases}
\overline{N_{4}^{0}}\left(\eta_{1},\eta_{2},\eta_{3}\right)\equiv N_{4}^{0}\left(2c\eta_{1}+c\eta_{2}+c\eta_{3},c\eta_{2}-c\eta_{3}\right)\\
\text{ and }\overline{N_{4}^{+}}\left(\eta_{1},\eta_{2},\eta_{3}\right)\equiv N_{4}^{+}\left(2\eta_{1}+\eta_{2}+\eta_{3}-\dfrac{1}{c}b_{1},c\eta_{2}-c\eta_{3}\right).
\end{cases}\label{eq:lopolop}
\end{equation}

Let us show that if $\sigma$ is sufficiently large, for all $\widetilde{M},$
the solution $\widetilde{N}=\left(\widetilde{N_{1}},\widetilde{N_{2}},\widetilde{N_{3}},\widetilde{N_{4}}\right)$
de $\left(\Sigma_{\sigma,\widetilde{M}}^{1}\right)$ where $\widetilde{N_{1}},\widetilde{N_{2}},\widetilde{N_{3}},\widetilde{N_{4}}$
are defined by (\ref{aoalal})-(\ref{eq:ayyeiz-1}), is non-negative
.

As the data $N_{1}^{-},\,N_{1}^{0},$ $N_{2}^{0},$ $N_{2}^{-},$
$N_{3}^{0}$, $N_{3}^{+}$, $N_{4}^{0}$, $N_{4}^{+}$ are non-negative,
it is sufficient that $Q_{i}^{\sigma}\left(\left|\widetilde{M}\right|\right)\geq0$,
$i=1,2,3,4.$ One has:
\begin{equation}
\begin{cases}
Q_{1}^{\sigma}\left(\left|\widetilde{M}\right|\right)=\sigma\left(\left|\widetilde{M_{1}}\right|+\left|\widetilde{M_{2}}\right|+\left|\widetilde{M_{3}}\right|\right)\left|\widetilde{M_{1}}\right|+2cS\left|\widetilde{M_{2}}\right|\left|\widetilde{M_{3}}\right|\\
\qquad\qquad+\left(\sigma-2cS\right)\left|\widetilde{M_{1}}\right|\left|\widetilde{M_{4}}\right|\\
Q_{2}^{\sigma}\left(\left|\widetilde{M}\right|\right)=\sigma\left(\left|\widetilde{M_{1}}\right|+\left|\widetilde{M_{2}}\right|+\left|\widetilde{M_{4}}\right|\right)\left|\widetilde{M_{2}}\right|+2cS\left|\widetilde{M_{1}}\right|\left|\widetilde{M_{4}}\right|\\
\qquad\qquad+\left(\sigma-2cS\right)\left|\widetilde{M_{2}}\right|\left|\widetilde{M_{3}}\right|\\
Q_{3}^{\sigma}\left(\left|\widetilde{M}\right|\right)=\sigma\left(\left|\widetilde{M_{1}}\right|+\left|\widetilde{M_{3}}\right|+\left|\widetilde{M_{4}}\right|\right)\left|\widetilde{M_{3}}\right|+2cS\left|\widetilde{M_{1}}\right|\left|\widetilde{M_{4}}\right|\\
\qquad\qquad+\left(\sigma-2cS\right)\left|\widetilde{M_{2}}\right|\left|\widetilde{M_{3}}\right|\\
Q_{4}^{\sigma}\left(\left|\widetilde{M}\right|\right)=\sigma\left(\left|\widetilde{M_{2}}\right|+\left|\widetilde{M_{3}}\right|+\left|\widetilde{M_{4}}\right|\right)\left|\widetilde{M_{4}}\right|+2cS\left|\widetilde{M_{2}}\right|\left|\widetilde{M_{3}}\right|\\
\qquad\qquad+\left(\sigma-2cS\right)\left|\widetilde{M_{1}}\right|\left|\widetilde{M_{4}}\right|.
\end{cases}
\end{equation}
From which we conclude that for $\sigma\geq2cS,$ the solution $\widetilde{N}=\left(\widetilde{N_{1}},\widetilde{N_{2}},\widetilde{N_{3}},\widetilde{N_{4}}\right)$
of $\left(\Sigma_{\sigma,\widetilde{M}}^{1}\right)$ is non-negative.
\end{proof}
We can thus consider the non-negative operator 
\begin{align}
\mathcal{T}^{\sigma}:C\left(\begin{array}{r}
\mathscr{P}'\end{array};\R^{4}\right) & \longrightarrow C\left(\begin{array}{r}
\mathscr{P}'\end{array};\R^{4}\right)\nonumber \\
\widetilde{M} & \longmapsto\widetilde{N}_{\widetilde{M}},\label{eq:ksoo}
\end{align}
where $\widetilde{N}_{\widetilde{M}}$ is the unique non-negative
solution of the problem $\left(\Sigma_{\sigma,\widetilde{M}}^{1}\right)$
for sufficiently large $\sigma.$

\subsection{Non-negativity theorem. }
\begin{thm}
:\label{thm::etaatyeyaiioa-1-1}

The solutions of the problem $\Sigma^{1}$ (eq.\ref{eq:ssdffz}-\ref{eq:yui})
are non-negative . 
\end{thm}
\begin{proof}
:

Let us verify that $\widetilde{N}=\left(\widetilde{N_{1}},\widetilde{N_{2}},\widetilde{N_{3}},\widetilde{N_{4}}\right)$
is a solution of $\Sigma^{1}$ if $\widetilde{N}$ is a fixed point
of the operateur $\mathcal{T}^{\sigma}$ (\ref{eq:ksoo}) for sufficiently
large $\sigma$.

We have $\widetilde{M}=\left(\widetilde{M_{1}},\widetilde{M_{2}},\widetilde{M_{3}},\widetilde{M_{4}}\right)\in C\left(\begin{array}{r}
\mathscr{P}'\end{array};\R^{4}\right)$ is a fixed point of $\mathcal{T}^{\sigma}$ if $\widetilde{N}_{\widetilde{M}}=\widetilde{M}$,
i.e. $\widetilde{M}$ is a solution of $\left(\Sigma_{\sigma,\widetilde{M}}^{1}\right)$
i.e. 
\begin{align}
\dfrac{\partial\widetilde{M_{1}}}{\partial\eta_{1}}+\sigma\rho\left(\left|\widetilde{M}\right|\right)\widetilde{M_{1}} & =Q_{1}^{\sigma}\left(\left|\widetilde{M}\right|\right),\:\left(\eta_{1},\eta_{2},\eta_{3}\right)\in\mathring{\overbrace{\mathscr{F}\left(\mathscr{P}\right)}}\equiv\mathring{\mathscr{P}'}\label{eq:ssdffz-1-1-1-1-1}\\
\dfrac{\partial\widetilde{M_{2}}}{\partial\eta_{2}}+\sigma\rho\left(\left|\widetilde{M}\right|\right)\widetilde{M_{2}} & =Q_{2}^{\sigma}\left(\left|\widetilde{M}\right|\right),\:\left(\eta_{1},\eta_{2},\eta_{3}\right)\in\mathring{\mathscr{P}'}\\
\dfrac{\partial\widetilde{M_{3}}}{\partial\eta_{3}}+\sigma\rho\left(\left|\widetilde{M}\right|\right)\widetilde{M_{3}} & =Q_{3}^{\sigma}\left(\left|\widetilde{M}\right|\right),\:\left(\eta_{1},\eta_{2},\eta_{3}\right)\in\mathring{\mathscr{P}'}\\
-\dfrac{\partial\widetilde{M_{4}}}{\partial\eta_{1}}+\dfrac{\partial\widetilde{M_{4}}}{\partial\eta_{2}}+\dfrac{\partial\widetilde{M_{4}}}{\partial\eta_{3}}+\sigma\rho\left(\left|\widetilde{M}\right|\right)\widetilde{M_{4}} & =Q_{4}^{\sigma}\left(\left|\widetilde{M}\right|\right),\:\left(\eta_{1},\eta_{2},\eta_{3}\right)\in\mathring{\mathscr{P}'}\label{eq:qqqqsa-1-1-1-1-1}
\end{align}
with the conditions (\ref{eq:dlz})-(\ref{eq:yui}); as $\widetilde{N}_{\widetilde{M}}=\widetilde{M}$
is positive i.e. $\left|\widetilde{M}\right|=\widetilde{M}$ for sufficiently
large $\sigma$, (\ref{eq:ssdffz-1-1-1-1-1})-(\ref{eq:qqqqsa-1-1-1-1-1})
means $\widetilde{M}$ is a solution of $\left(\Sigma_{\sigma}^{1}\right)$
which is equivalent to $\Sigma^{1}.$ As $\mathcal{T}^{\sigma}$ is
non-negative, so are its fixed points.
\end{proof}

\section{Existence and uniqueness of bounded solution\label{sec:Resolution}}

We shall define an operator, the fixed points of which, are the solutions
of the problem $\Sigma^{1}$ (eq.\ref{eq:ssdffz}-\ref{eq:yui}) and
establish the existence of the fixed points by using the following
Schauder's theorem ( \cite{17}, p.25, Theorem 4.1.1). 
\begin{thm}
\label{sccssq}( Schauder\cite{17} ) Let $\mathcal{M}$ be a non-empty
convex subset of a normed space $\mathscr{B}.$ Let $\mathcal{T}$
be a continuous compact mapping from $\mathcal{M}$ into $\mathcal{M}$
. Then $\mathcal{T}$ has a fixed point. 
\end{thm}

\subsection{Fixed point problem\label{subsec:Fixed-point-problem}}

Let $\widetilde{M}=\left(\widetilde{M_{1}},\widetilde{M_{2}},\widetilde{M_{3}},\widetilde{M_{4}}\right)$
be a fixed 4-tuple of continuous functions from $\begin{array}{r}
\mathscr{P}'\end{array}$ into $\R$. Let us consider, the following decoupled system $\left(\Sigma_{\widetilde{M}}^{1}\right)$
defined by (\ref{eq:ssdffz-1-1-1})-(\ref{eq:qqqqsa-1-1-1}) 
\begin{align}
\dfrac{\partial\widetilde{N_{1}}}{\partial\eta_{1}} & =Q(\widetilde{M}),\:\left(\eta_{1},\eta_{2},\eta_{3}\right)\in\mathring{\overbrace{\mathscr{F}\left(\mathscr{P}\right)}}\equiv\mathring{\mathscr{P}'}\label{eq:ssdffz-1-1-1}\\
\dfrac{\partial\widetilde{N_{2}}}{\partial\eta_{2}} & =-Q(\widetilde{M}),\:\left(\eta_{1},\eta_{2},\eta_{3}\right)\in\mathring{\mathscr{P}'}\\
\dfrac{\partial\widetilde{N_{3}}}{\partial\eta_{3}} & =-Q(\widetilde{M}),\:\left(\eta_{1},\eta_{2},\eta_{3}\right)\in\mathring{\mathscr{P}'}\\
-\dfrac{\partial\widetilde{N_{4}}}{\partial\eta_{1}}+\dfrac{\partial\widetilde{N_{4}}}{\partial\eta_{2}}+\dfrac{\partial\widetilde{N_{4}}}{\partial\eta_{3}} & =Q(\widetilde{M}),\:\left(\eta_{1},\eta_{2},\eta_{3}\right)\in\mathring{\mathscr{P}'}\label{eq:qqqqsa-1-1-1}
\end{align}
with the conditions (\ref{eq:dlz})-(\ref{eq:yui}). Then it follows
from the resolution made in the proof of proposition (\ref{prop::azeeza-1})
that the problem $\left(\Sigma_{\widetilde{M}}^{1}\right)$ has an
unique continuous solution $\widetilde{N}=\left(\widetilde{N_{1}},\widetilde{N_{2}},\widetilde{N_{3}},\widetilde{N_{4}}\right)$
defined by:

\begin{multline}
\widetilde{N_{1}}\left(\eta_{1},\eta_{2},\eta_{3}\right)\\
=\left(\int_{-\eta_{2}-\eta_{3}}^{\eta_{1}}Q\left(\widetilde{M}\right)\left(s,\eta_{2},\eta_{3}\right)ds+\overline{N_{1}^{0}}\left(\eta_{2},\eta_{3}\right)\right)\cdot\mathbb{I}_{-c\eta_{2}-c\eta_{3}\geq a_{1}}\left(\eta_{1},\eta_{2},\eta_{3}\right)\\
+\left(\int_{\frac{1}{c}a_{1}}^{\eta_{1}}Q\left(\widetilde{M}\right)\left(s,\eta_{2},\eta_{3}\right)ds+\overline{N_{1}^{-}}\left(\eta_{2},\eta_{3}\right)\right)\cdot\mathbb{I}_{-c\eta_{2}-c\eta_{3}\leq a_{1}}\left(\eta_{1},\eta_{2},\eta_{3}\right)\label{aoalal-2-1}
\end{multline}
\begin{multline}
\widetilde{N_{2}}\left(\eta_{1},\eta_{2},\eta_{3}\right)\\
=\left(\int_{-\eta_{1}-\eta_{3}}^{\eta_{2}}-Q\left(\widetilde{M}\right)\left(\eta_{1},s,\eta_{3}\right)ds+\overline{N_{2}^{0}}\left(\eta_{1},\eta_{3}\right)\right)\cdot\mathbb{I}_{-c\eta_{1}-2c\eta_{3}\geq a_{2}}\left(\eta_{1},\eta_{2},\eta_{3}\right)\\
+\left({\displaystyle \int_{\eta_{3}+\frac{a_{2}}{c}}^{\eta_{2}}}-Q\left(\widetilde{M}\right)\left(\eta_{1},s,\eta_{3}\right)ds+\overline{N_{2}^{-}}\left(\eta_{1},\eta_{3}\right)\right)\cdot\mathbb{I}_{-c\eta_{1}-2c\eta_{3}\leq a_{2}}\left(\eta_{1},\eta_{2},\eta_{3}\right)\label{eq:ppmms-2-1}
\end{multline}
\begin{multline}
\widetilde{N_{3}}\left(\eta_{1},\eta_{2},\eta_{3}\right)\\
=\left({\displaystyle \int_{-\eta_{1}-\eta_{2}}^{\eta_{3}}}-Q\left(\widetilde{M}\right)\left(\eta_{1},\eta_{2},s\right)ds+\overline{N_{3}^{0}}\left(\eta_{1},\eta_{2}\right)\right)\cdot\mathbb{I}_{c\eta_{1}+2c\eta_{2}\leq b_{2}}\left(\eta_{1},\eta_{2},\eta_{3}\right)\\
+\left({\displaystyle \int_{\eta_{2}-\frac{b_{2}}{c}}^{\eta_{3}}}-Q\left(\widetilde{M}\right)\left(\eta_{1},\eta_{2},s\right)ds+\overline{N_{3}^{+}}\left(\eta_{1},\eta_{2}\right)\right)\cdot\mathbb{I}_{c\eta_{1}+2c\eta_{2}\geq b_{2}}\left(\eta_{1},\eta_{2},\eta_{3}\right)\label{eq:oslslo-2-1}
\end{multline}
\begin{multline}
\widetilde{N_{4}}\left(\eta_{1},\eta_{2},\eta_{3}\right)=\mathbb{I}_{2c\eta_{1}+c\eta_{2}+c\eta_{3}\leq b_{1}}\left(\eta_{1},\eta_{2},\eta_{3}\right)\\
\cdot\left[\int_{0}^{\eta_{1}+\eta_{2}+\eta_{3}}\right.Q\left(\widetilde{M}\right)\left(-s+2\eta_{1}+\eta_{2}+\eta_{3};s-\eta_{1}-\eta_{3};s-\eta_{1}-\eta_{2}\right)ds+\\
\left.\overline{N_{4}^{0}}\left(\eta_{1},\eta_{2},\eta_{3}\right)\right]+\mathbb{I}_{2c\eta_{1}+c\eta_{2}+c\eta_{3}\geq b_{1}}\left(\eta_{1},\eta_{2},\eta_{3}\right)\\
\cdot\left[\int_{0}^{\left(-\eta_{1}+\frac{1}{c}b_{1}\right)}\right.Q\left(\widetilde{M}\right)\left(-s+\frac{1}{c}b_{1},s+\eta_{1}+\eta_{2}-\dfrac{1}{c}b_{1},s+\eta_{1}+\eta_{3}-\dfrac{1}{c}b_{1}\right)ds\\
\left.+\overline{N_{4}^{+}}\left(\eta_{1},\eta_{2},\eta_{3}\right)\right].\label{eq:ayyeiz-1-1-2-1}
\end{multline}

We can thus define the following operator 
\begin{align}
\mathcal{T}:C\left(\begin{array}{r}
\mathscr{P}'\end{array};\R^{4}\right) & \longrightarrow C\left(\begin{array}{r}
\mathscr{P}'\end{array};\R^{4}\right)\nonumber \\
\widetilde{M} & \longmapsto\mathcal{T}\left(\widetilde{M}\right)=\left(\mathcal{T}_{i}\left(\widetilde{M}\right)\right)_{i=1}^{4}\label{eq:kozep}
\end{align}
where $\mathcal{T}\left(\widetilde{M}\right)=\left(\mathcal{T}_{i}\left(\widetilde{M}\right)\right)_{i=1}^{4}$
is the unique solution of the problem $\left(\Sigma_{\widetilde{M}}^{1}\right).$

It immediatly follows from (\ref{eq:ssdffz-1-1-1})-(\ref{eq:qqqqsa-1-1-1})
and (\ref{eq:ssdffz})-(\ref{eq:yui}) that the solutions of $\Sigma^{1}$
are the fixed points of $\mathcal{T}.$

\subsection{Continuity of the operator of the fixed point problem}
\begin{prop}
\label{prop::mppp}The operator $\mathcal{\mathcal{T}}$ ( \ref{eq:kozep})
is continous. 
\end{prop}
\begin{proof}
:Relation \ref{aoalal-2-1} gives for $\widetilde{M},\widetilde{N}\in C\left(\mathscr{P}',\R^{4}\right):$
\begin{multline}
\left\Vert \mathcal{T}_{1}\left(\widetilde{M}\right)-\mathcal{T}_{1}\left(\widetilde{N}\right)\right\Vert _{\infty}\leq\max\Biggl\{\sup_{\left(\eta_{1},\eta_{2},\eta_{3}\right)\in\mathscr{P}'}\Biggl|\int_{-\eta_{2}-\eta_{3}}^{\eta_{1}}\left[Q\left(\widetilde{M}\right)-Q\left(\widetilde{N}\right)\right]\left(s,\eta_{2},\eta_{3}\right)ds\Biggr|;\\
\sup_{\left(\eta_{1},\eta_{2},\eta_{3}\right)\in\mathscr{P}'}\Biggl|\int_{\frac{1}{c}a_{1}}^{\eta_{1}}\left[Q\left(\widetilde{M}\right)-Q\left(\widetilde{N}\right)\right]\left(s,\eta_{2},\eta_{3}\right)ds\Biggr|\Biggl\}.\label{aoalal-2-1-1-1-1-1-1}
\end{multline}
But \ref{eq:solls} yields 
\begin{multline}
Q\left(\widetilde{M}\right)-Q\left(\widetilde{N}\right)=2cS\left(\widetilde{M_{2}}-\widetilde{N_{2}}\right)\widetilde{M_{3}}+2cS\widetilde{N_{2}}\left(\widetilde{M_{3}}-\widetilde{N_{3}}\right)\\
-2cS\left(\widetilde{M_{1}}-\widetilde{N_{1}}\right)\widetilde{M_{4}}-2cS\widetilde{N_{1}}\left(\widetilde{M_{4}}-\widetilde{N_{4}}\right),
\end{multline}
hence 
\begin{multline*}
\left\Vert Q\left(\widetilde{M}\right)-Q\left(\widetilde{N}\right)\right\Vert _{\infty}\\
\leq2cS\left\Vert \widetilde{M_{2}}-\widetilde{N_{2}}\right\Vert _{\infty}\left\Vert \widetilde{M_{3}}\right\Vert _{\infty}+2cS\left\Vert \widetilde{N_{2}}\right\Vert _{\infty}\left\Vert \widetilde{M_{3}}-\widetilde{N_{3}}\right\Vert _{\infty}\\
+2cS\left\Vert \widetilde{M_{1}}-\widetilde{N_{1}}\right\Vert _{\infty}\left\Vert \widetilde{M_{4}}\right\Vert _{\infty}+2cS\left\Vert \widetilde{N_{1}}\right\Vert _{\infty}\left\Vert \widetilde{M_{4}}-\widetilde{N_{4}}\right\Vert _{\infty},
\end{multline*}
and 
\begin{multline*}
\left\Vert Q\left(\widetilde{M}\right)-Q\left(\widetilde{N}\right)\right\Vert _{\infty}\leq4cS\left\Vert \widetilde{M}-\widetilde{N}\right\Vert \left\Vert \widetilde{M}\right\Vert +4cS\left\Vert \widetilde{N}\right\Vert \left\Vert \widetilde{M}-\widetilde{N}\right\Vert \\
\left\Vert Q\left(\widetilde{M}\right)-Q\left(\widetilde{N}\right)\right\Vert _{\infty}\leq4cS\left(\left\Vert \widetilde{M}\right\Vert +\left\Vert \widetilde{N}\right\Vert \right)\left\Vert \widetilde{M}-\widetilde{N}\right\Vert ;
\end{multline*}

hence \ref{aoalal-2-1-1-1-1-1-1} implies 
\begin{multline}
\left\Vert \mathcal{T}_{1}\left(\widetilde{M}\right)-\mathcal{T}_{1}\left(\widetilde{N}\right)\right\Vert _{\infty}\leq\\
\max\Biggl\{\left(\sup_{\left(\eta_{1},\eta_{2},\eta_{3}\right)\in\mathscr{P}'}\left(\eta_{1}+\eta_{2}+\eta_{3}\right)\right)\cdot4cS\left(\left\Vert \widetilde{M}\right\Vert +\left\Vert \widetilde{N}\right\Vert \right)\left\Vert \widetilde{M}-\widetilde{N}\right\Vert ;\\
\left(\sup_{\left(\eta_{1},\eta_{2},\eta_{3}\right)\in\mathscr{P}'}\left(\eta_{1}-\frac{1}{c}a_{1}\right)\right)\cdot4cS\left(\left\Vert \widetilde{M}\right\Vert +\left\Vert \widetilde{N}\right\Vert \right)\left\Vert \widetilde{M}-\widetilde{N}\right\Vert \Biggl\}.\label{aoalal-2-1-1-1-1-1-1-1-3}
\end{multline}
Now $0\leq\eta_{1}+\eta_{2}+\eta_{3}\leq T$ and $a_{1}\leq c\eta_{1}\leq b_{1};$
hence $0\leq\eta_{1}-{\displaystyle \frac{1}{c}}a_{1}\leq{\displaystyle \frac{1}{c}}\left(b_{1}-a_{1}\right);$
from which 
\begin{multline}
\left\Vert \mathcal{T}_{1}\left(\widetilde{M}\right)-\mathcal{T}_{1}\left(\widetilde{N}\right)\right\Vert _{\infty}\\
\leq\max\Biggl\{ T,{\displaystyle \frac{1}{c}}\left(b_{1}-a_{1}\right)\Biggl\}\cdot4cS\left(\left\Vert \widetilde{M}\right\Vert +\left\Vert \widetilde{N}\right\Vert \right)\left\Vert \widetilde{M}-\widetilde{N}\right\Vert .\label{aoalal-2-1-1-1-1-1-1-1-1-2}
\end{multline}
 Similarly from \ref{eq:ppmms-2-1} we have 
\begin{multline}
\left\Vert \mathcal{T}_{2}\left(\widetilde{M}\right)-\mathcal{T}_{2}\left(\widetilde{N}\right)\right\Vert _{\infty}\leq\max\Biggl\{\sup_{\left(\eta_{1},\eta_{2},\eta_{3}\right)\in\mathscr{P}'}\Biggl|\int_{-\eta_{2}-\eta_{3}}^{\eta_{1}}\left[Q\left(\widetilde{M}\right)-Q\left(\widetilde{N}\right)\right]\left(s,\eta_{2},\eta_{3}\right)ds\Biggr|;\\
\sup_{\left(\eta_{1},\eta_{2},\eta_{3}\right)\in\mathscr{P}'}\Biggl|{\displaystyle \int_{\eta_{3}+\frac{a_{2}}{c}}^{\eta_{2}}}\left[Q\left(\widetilde{M}\right)-Q\left(\widetilde{N}\right)\right]\left(s,\eta_{2},\eta_{3}\right)ds\Biggr|\Biggl\}\label{aoalal-2-1-1-1-1-1-2-1}
\end{multline}
and 
\begin{multline}
\left\Vert \mathcal{T}_{2}\left(\widetilde{M}\right)-\mathcal{T}_{2}\left(\widetilde{N}\right)\right\Vert _{\infty}\leq\\
\max\Biggl\{\left(\sup_{\left(\eta_{1},\eta_{2},\eta_{3}\right)\in\mathscr{P}'}\left(\eta_{1}+\eta_{2}+\eta_{3}\right)\right)\cdot4cS\left(\left\Vert \widetilde{M}\right\Vert +\left\Vert \widetilde{N}\right\Vert \right)\left\Vert \widetilde{M}-\widetilde{N}\right\Vert ;\\
\left(\sup_{\left(\eta_{1},\eta_{2},\eta_{3}\right)\in\mathscr{P}'}\left(\eta_{2}-\eta_{3}-\dfrac{a_{2}}{c}\right)\right)\cdot4cS\left(\left\Vert \widetilde{M}\right\Vert +\left\Vert \widetilde{N}\right\Vert \right)\left\Vert \widetilde{M}-\widetilde{N}\right\Vert \Biggl\}.\label{aoalal-2-1-1-1-1-1-1-1-2-1}
\end{multline}
But $0\leq\eta_{1}+\eta_{2}+\eta_{3}\leq T$ and $a_{2}\leq c\eta_{2}-c\eta_{3}\leq b_{2};$
hence $0\leq\eta_{2}-\eta_{3}-\dfrac{a_{2}}{c}\leq{\displaystyle \frac{1}{c}}\left(b_{2}-a_{2}\right);$
thus 
\begin{align}
\left\Vert \mathcal{T}_{2}\left(\widetilde{M}\right)-\mathcal{T}_{2}\left(\widetilde{N}\right)\right\Vert _{\infty} & \leq\max\Biggl\{ T,{\displaystyle \frac{1}{c}}\left(b_{2}-a_{2}\right)\Biggl\}\cdot4cS\left(\left\Vert \widetilde{M}\right\Vert +\left\Vert \widetilde{N}\right\Vert \right)\left\Vert \widetilde{M}-\widetilde{N}\right\Vert .\label{aoalal-2-1-1-1-1-1-1-1-1-1-2}
\end{align}
{*}{*}{*}{*} Similarly \ref{eq:oslslo-2-1} and \ref{eq:ayyeiz-1-1-2-1}
yield 
\begin{align}
\left\Vert \mathcal{T}_{3}\left(\widetilde{M}\right)-\mathcal{T}_{3}\left(\widetilde{N}\right)\right\Vert _{\infty} & \leq\max\Biggl\{ T,{\displaystyle \frac{1}{c}}\left(b_{2}-a_{2}\right)\Biggl\}\cdot4cS\left(\left\Vert \widetilde{M}\right\Vert +\left\Vert \widetilde{N}\right\Vert \right)\left\Vert \widetilde{M}-\widetilde{N}\right\Vert .\label{aoalal-2-1-1-1-1-1-1-1-1-1-1-2}
\end{align}
and 
\begin{align}
\left\Vert \mathcal{T}_{4}\left(\widetilde{M}\right)-\mathcal{T}_{4}\left(\widetilde{N}\right)\right\Vert _{\infty} & \leq\max\Biggl\{ T,{\displaystyle \frac{1}{c}}\left(b_{1}-a_{1}\right)\Biggl\}\cdot4cS\left(\left\Vert \widetilde{M}\right\Vert +\left\Vert \widetilde{N}\right\Vert \right)\left\Vert \widetilde{M}-\widetilde{N}\right\Vert .\label{aoalal-2-1-1-1-1-1-1-1-1-1-1-1-1}
\end{align}
{*}{*}{*}{*} Now from (\ref{aoalal-2-1-1-1-1-1-1-1-1-2})-(\ref{aoalal-2-1-1-1-1-1-1-1-1-1-2})-
(\ref{aoalal-2-1-1-1-1-1-1-1-1-1-1-2})-(\ref{aoalal-2-1-1-1-1-1-1-1-1-1-1-1-1})
and

$\left\Vert \mathcal{T}\left(\widetilde{M}\right)-\mathcal{T}\left(\widetilde{N}\right)\right\Vert =\max_{1\leq i\leq4}\left\Vert \mathcal{T}_{i}\left(\widetilde{M}\right)-\mathcal{T}_{i}\left(\widetilde{N}\right)\right\Vert _{\infty}$we
have 
\begin{multline}
\left\Vert \mathcal{T}\left(\widetilde{M}\right)-\mathcal{T}\left(\widetilde{N}\right)\right\Vert \leq\\
\underbrace{\max\Biggl\{ T,{\displaystyle \frac{1}{c}}\left(b_{1}-a_{1}\right),{\displaystyle \frac{1}{c}}\left(b_{2}-a_{2}\right)\Biggl\}\cdot4cS}_{\equiv p'}\left(\left\Vert \widetilde{M}\right\Vert +\left\Vert \widetilde{N}\right\Vert \right)\left\Vert \widetilde{M}-\widetilde{N}\right\Vert \label{eq:lsoopa-2}
\end{multline}
For a fixed $\widetilde{N}$ and a fixed $R>0,$ $\left\Vert \widetilde{M}-\widetilde{N}\right\Vert \leq R\implies\left\Vert \widetilde{M}\right\Vert +\left\Vert \widetilde{N}\right\Vert \leq2\left\Vert \widetilde{N}\right\Vert +R$
\\
 thus forall $\varepsilon>0$, 
\[
\left\Vert \widetilde{M}-\widetilde{N}\right\Vert \leq\min\left\{ R;\dfrac{\varepsilon}{p'\left(2\left\Vert \widetilde{N}\right\Vert +R\right)}\right\} \implies\left\Vert \mathcal{T}\left(\widetilde{M}\right)-\mathcal{T}\left(\widetilde{N}\right)\right\Vert \leq\varepsilon.
\]

We deduce that $\mathcal{T}$ is continuous.
\end{proof}

\subsection{Convex set on which the operator is compact.}
\begin{prop}
\label{cor::mzppzo}Suppose $\widetilde{M}=\left(\widetilde{M_{1}},\widetilde{M_{2}},\widetilde{M_{3}},\widetilde{M_{4}}\right)\in C\left(\begin{array}{r}
\mathscr{P}'\end{array};\R^{4}\right)$ such that $\dfrac{\partial\widetilde{M_{i}}}{\partial\eta_{1}},\dfrac{\partial\widetilde{M_{i}}}{\partial\eta_{2}},\dfrac{\partial\widetilde{M_{i}}}{\partial\eta_{3}}$
are defined in $\mathring{\mathscr{P}'},$ except possibly on a finite
number of planes, and are continuous and bounded forall $i=1,2,3,4.$
Then all the derivatives $\dfrac{\partial\mathcal{T}_{i}\left(\widetilde{M}\right)}{\partial\eta_{j}},$$\left(j=1,2,3\right)$,
$\left(i=1,2,3,4\right)$ are defined in $\mathring{\mathscr{P}'},$
except possibly on a finite number of planes, and are continuous and
bounded. 
\end{prop}
\begin{proof}
: It follows immediately from the formula (\ref{aoalal-2-1})-(\ref{eq:ayyeiz-1-1-2-1})
as the derivatives of both the integrand and data are defined, continuous
and bounded.
\end{proof}
Let $E$ be the sub-space of $C\left(\begin{array}{r}
\mathscr{P}'\end{array};\R\right)$ consisting of functions $u$ that are continuous on $\mathscr{P}'$
such that $\dfrac{\partial u}{\partial\eta_{j}},j=1,2,3$ are defined
in $\mathring{\mathscr{P}'}$ except possibly on a finite number of
planes, and are continuous and bounded. The above proposition states
that $\forall\widetilde{M}\in E^{4}\subset C\left(\begin{array}{r}
\mathscr{P}'\end{array};\R^{4}\right),\mathcal{T}\left(\widetilde{M}\right)\in E^{4}.$ 

\begin{prop}
:\label{prop::odlp}

Let us set for all $R>0,$ h
\begin{equation}
\mathscr{M}_{R}\equiv\left\{ \widetilde{N}\in E^{4}:\mathscr{N}\left(\widetilde{N}\right)\equiv\max\left\{ \left\Vert \widetilde{N}\right\Vert ,\left\Vert \dfrac{\partial\widetilde{N}}{\partial\eta_{1}}\right\Vert ,\left\Vert \dfrac{\partial\widetilde{N}}{\partial\eta_{2}}\right\Vert ,\left\Vert \dfrac{\partial\widetilde{N}}{\partial\eta_{3}}\right\Vert \right\} \leq R\right\} .\label{eq:poiop}
\end{equation}

The $\mathscr{M}_{R}$ is a non-empty convex subset of $C\left(\begin{array}{r}
\mathscr{P}'\end{array};\R^{4}\right)$ .
\end{prop}
\begin{proof}
$\mathscr{M}_{R}$ is non-empty for it contains the zero function.
For $\widetilde{M},\widetilde{N}\in\mathscr{M}_{R},$ and $\lambda\in\R$
such that $0\leq\lambda\leq1$ we have $\lambda\widetilde{M}+\left(1-\lambda\right)\widetilde{N}\in E^{4}.$
Moreover $\mathscr{N}\left(\lambda\widetilde{M}+\left(1-\lambda\right)\widetilde{N}\right)\leq R$
follows from triangular inequality.
\end{proof}
\begin{prop}
\label{prop::opps}The operator $\mathcal{T}$ is compact on $\mathscr{M}_{R}$
forall $R>0.$
\end{prop}
\begin{proof}
First, we prove that $\mathcal{T}\left(\mathcal{\mathscr{M}}_{R}\right)$
is bounded in $C\left(\begin{array}{r}
\mathscr{P}'\end{array};\R^{4}\right).$ From 
\begin{equation}
Q\left(\widetilde{M}\right)={\displaystyle 2cS\left(\widetilde{M_{2}}\widetilde{M_{3}}-\widetilde{M_{1}}\widetilde{M_{4}}\right)}
\end{equation}
we infer 
\begin{equation}
\left\Vert Q\left(\widetilde{M}\right)\right\Vert _{\infty}=\left\Vert 2cS\left(\widetilde{M_{2}}\widetilde{M_{3}}-\widetilde{M_{1}}\widetilde{M_{4}}\right)\right\Vert _{\infty}\leq4cS\left(\mathscr{N}\left(\widetilde{M}\right)\right)^{2}.\label{eq:kooll}
\end{equation}

So (\ref{aoalal-2-1})-(\ref{eq:ayyeiz-1-1-2-1}) yield for all $\widetilde{M}\in\mathcal{\mathscr{M}}_{R}:$ 

\begin{multline}
\left\Vert \mathcal{T}_{1}\left(\widetilde{M}\right)\right\Vert _{\infty}\leq4cS\cdot\max\Biggl\{ T;{\displaystyle \frac{1}{c}}\left(b_{1}-a_{1}\right)\Biggr\}\left(\mathscr{N}\left(\widetilde{M}\right)\right)^{2}\\
+\max\Biggl\{\left\Vert \overline{N_{1}^{0}}\right\Vert _{1};\left\Vert \overline{N_{1}^{-}}\right\Vert _{1}\Biggr\};
\end{multline}
\begin{multline}
\left\Vert \mathcal{T}_{2}\left(\widetilde{M}\right)\right\Vert _{\infty}\leq4cS\cdot\max\Biggl\{ T;{\displaystyle \frac{1}{c}}\left(b_{2}-a_{2}\right)\Biggr\}\left(\mathscr{N}\left(\widetilde{M}\right)\right)^{2}\\
+\max\Biggl\{\left\Vert \overline{N_{2}^{0}}\right\Vert _{1};\left\Vert \overline{N_{2}^{-}}\right\Vert _{1}\Biggr\}\label{eq:ksoloa-1}
\end{multline}

\begin{multline}
\left\Vert \mathcal{T}_{3}\left(\widetilde{M}\right)\right\Vert _{\infty}\leq4cS\cdot\max\Biggl\{ T;{\displaystyle \frac{1}{c}}\left(b_{2}-a_{2}\right)\Biggr\}\left(\mathscr{N}\left(\widetilde{M}\right)\right)^{2}\\
+\max\Biggl\{\left\Vert \overline{N_{3}^{0}}\right\Vert _{1};\left\Vert \overline{N_{3}^{+}}\right\Vert _{1}\Biggr\}\label{eq:ksoloa-1-1}
\end{multline}

\begin{multline}
\left\Vert \mathcal{T}_{4}\left(\widetilde{M}\right)\right\Vert _{\infty}\leq4cS\cdot\max\Biggl\{ T;{\displaystyle \frac{1}{c}}\left(b_{1}-a_{1}\right)\Biggr\}\left(\mathscr{N}\left(\widetilde{M}\right)\right)^{2}\\
+\max\Biggl\{\left\Vert \overline{N_{4}^{0}}\right\Vert _{1};\left\Vert \overline{N_{4}^{+}}\right\Vert _{1}\Biggr\}.\label{eq:ksoloa-1-1-1}
\end{multline}

Then for all $\widetilde{M}\in\mathcal{\mathscr{M}}_{R}:$ 
\begin{multline}
\left\Vert \mathcal{T}\left(\widetilde{M}\right)\right\Vert \leq4cS\cdot\max\Biggl\{ T;{\displaystyle \frac{1}{c}}\left(b_{1}-a_{1}\right);{\displaystyle \frac{1}{c}}\left(b_{2}-a_{2}\right)\Biggr\} R^{2}+\\
\max_{1\leq i\leq4}\Biggl\{\left\Vert \overline{N_{i}^{0}}\right\Vert _{1};\left\Vert \overline{N_{1}^{-}}\right\Vert _{1};\left\Vert \overline{N_{2}^{-}}\right\Vert _{1};\left\Vert \overline{N_{3}^{+}}\right\Vert _{1};\left\Vert \overline{N_{4}^{+}}\right\Vert _{1}\Biggr\}\equiv R_{1}.\label{eq:lospp}
\end{multline}
 Second, we prove that $\mathcal{T}\left(\mathcal{\mathscr{M}}_{R}\right)$
is equicontinuous in $C\left(\begin{array}{r}
\mathscr{P}'\end{array};\R^{4}\right).$ From 
\begin{equation}
{\displaystyle \dfrac{\partial Q\left(\widetilde{M}\right)}{\partial\eta_{j}}=2cS\left(\dfrac{\partial\widetilde{M_{2}}}{\partial\eta_{j}}\widetilde{M_{3}}+\widetilde{M_{2}}\dfrac{\partial\widetilde{M_{3}}}{\partial\eta_{j}}-\dfrac{\partial\widetilde{M_{1}}}{\partial\eta_{j}}\widetilde{M_{4}}-\widetilde{M_{1}}\dfrac{\partial\widetilde{M_{4}}}{\partial\eta_{j}}\right)}
\end{equation}
we infer 

\begin{align}
\left\Vert \dfrac{\partial Q\left(\widetilde{M}\right)}{\partial\eta_{j}}\right\Vert _{\infty} & \leq2cS\cdot4\left(\mathscr{N}\left(\widetilde{M}\right)\right)^{2}\leq8cS\left(\mathscr{N}\left(\widetilde{M}\right)\right)^{2}.\label{eq:lopoi}
\end{align}
On one hand we have
\begin{equation}
\left\Vert \dfrac{\partial\mathcal{T}_{i}\left(\widetilde{M}\right)}{\partial\eta_{i}}\right\Vert _{\infty}\leq4cs\left(\mathscr{N}\left(\widetilde{M}\right)\right)^{2},i=1,2,3.\label{eq:ikiioz}
\end{equation}

On other hand, by derivation in the equations (\ref{aoalal-2-1})-(\ref{eq:ayyeiz-1-1-2-1}),
taking the norm and taking into account (\ref{eq:lopoi}), we have
\begin{multline}
\left\Vert \dfrac{\partial\mathcal{T}_{1}\left(\widetilde{M}\right)}{\partial\eta_{2}}\right\Vert _{\infty}\leq4cs\Biggl(1+2\cdot\max\Biggl\{ T;{\displaystyle \frac{1}{c}}\left(b_{1}-a_{1}\right)\Biggr\}\Biggr)\left(\mathscr{N}\left(\widetilde{M}\right)\right)^{2}\\
+\max\Biggl\{\left\Vert \overline{N_{1}^{0}}\right\Vert _{1};\left\Vert \overline{N_{1}^{-}}\right\Vert _{1}\Biggr\}.\label{eq:sloos}
\end{multline}

\begin{multline}
\left\Vert \dfrac{\partial\mathcal{T}_{1}\left(\widetilde{M}\right)}{\partial\eta_{3}}\right\Vert _{\infty}\leq4cs\Biggl(1+2\cdot\max\Biggl\{ T;{\displaystyle \frac{1}{c}}\left(b_{1}-a_{1}\right)\Biggr\}\Biggr)\left(\mathscr{N}\left(\widetilde{M}\right)\right)^{2}\\
+\max\Biggl\{\left\Vert \overline{N_{1}^{0}}\right\Vert _{1};\left\Vert \overline{N_{1}^{-}}\right\Vert _{1}\Biggr\};\label{eq:qmpql}
\end{multline}

\begin{multline}
\left\Vert \dfrac{\partial\mathcal{T}_{2}\left(\widetilde{M}\right)}{\partial\eta_{1}}\right\Vert _{\infty}\leq4cs\Biggl(1+2\cdot\max\Biggl\{ T;{\displaystyle \frac{1}{c}}\left(b_{2}-a_{2}\right)\Biggr\}\Biggr)\left(\mathscr{N}\left(\widetilde{M}\right)\right)^{2}\\
+\max\Biggl\{\left\Vert \overline{N_{2}^{0}}\right\Vert _{1};\left\Vert \overline{N_{2}^{-}}\right\Vert _{1}\Biggr\};\label{eq:qmpql-1}
\end{multline}
\begin{multline}
\left\Vert \dfrac{\partial\mathcal{T}_{2}\left(\widetilde{M}\right)}{\partial\eta_{3}}\right\Vert _{\infty}\leq4cs\Biggl(1+2\cdot\max\Biggl\{ T;{\displaystyle \frac{1}{c}}\left(b_{2}-a_{2}\right)\Biggr\}\Biggr)\left(\mathscr{N}\left(\widetilde{M}\right)\right)^{2}\\
+\max\Biggl\{\left\Vert \overline{N_{2}^{0}}\right\Vert _{1};\left\Vert \overline{N_{2}^{-}}\right\Vert _{1}\Biggr\};\label{eq:qmpql-1-1}
\end{multline}
\begin{multline}
\left\Vert \dfrac{\partial\mathcal{T}_{3}\left(\widetilde{M}\right)}{\partial\eta_{1}}\right\Vert _{\infty}\leq4cs\Biggl(1+2\cdot\max\Biggl\{ T;{\displaystyle \frac{1}{c}}\left(b_{2}-a_{2}\right)\Biggr\}\Biggr)\left(\mathscr{N}\left(\widetilde{M}\right)\right)^{2}\\
+\max\Biggl\{\left\Vert \overline{N_{3}^{0}}\right\Vert _{1};\left\Vert \overline{N_{3}^{+}}\right\Vert _{1}\Biggr\};\label{eq:qmpql-1-1-1}
\end{multline}
\begin{multline}
\left\Vert \dfrac{\partial\mathcal{T}_{3}\left(\widetilde{M}\right)}{\partial\eta_{2}}\right\Vert _{\infty}\leq4cs\Biggl(1+2\cdot\max\Biggl\{ T;{\displaystyle \frac{1}{c}}\left(b_{2}-a_{2}\right)\Biggr\}\Biggr)\left(\mathscr{N}\left(\widetilde{M}\right)\right)^{2}\\
+\max\Biggl\{\left\Vert \overline{N_{3}^{0}}\right\Vert _{1};\left\Vert \overline{N_{3}^{+}}\right\Vert _{1}\Biggr\};\label{eq:qmpql-1-1-1-1}
\end{multline}
\begin{multline}
\left\Vert \dfrac{\partial\mathcal{T}_{4}\left(\widetilde{M}\right)}{\partial\eta_{1}}\right\Vert _{\infty}\leq4cs\left(1+2\cdot\max\Biggl\{4T;{\displaystyle \frac{2}{c}}\left(b_{1}-a_{1}\right)\Biggr\}\right)\left(\mathscr{N}\left(\widetilde{M}\right)\right)^{2}\\
+\max\Biggl\{\left\Vert \overline{N_{4}^{0}}\right\Vert _{1};\left\Vert \overline{N_{4}^{+}}\right\Vert _{1}\Biggr\};\label{eq:aokkao}
\end{multline}

\begin{multline}
\left\Vert \dfrac{\partial\mathcal{T}_{4}\left(\widetilde{M}\right)}{\partial\eta_{2}}\right\Vert _{\infty}\leq4cs\left(1+2\cdot\max\Biggl\{2T;{\displaystyle \frac{1}{c}}\left(b_{1}-a_{1}\right)\Biggr\}\right)\left(\mathscr{N}\left(\widetilde{M}\right)\right)^{2}\\
+\max\Biggl\{\left\Vert \overline{N_{4}^{0}}\right\Vert _{1};\left\Vert \overline{N_{4}^{+}}\right\Vert _{1}\Biggr\};\label{eq:koalla}
\end{multline}
\begin{multline}
\left\Vert \dfrac{\partial\mathcal{T}_{4}\left(\widetilde{M}\right)}{\partial\eta_{3}}\right\Vert _{\infty}\leq4cs\left(1+2\cdot\max\Biggl\{2T;{\displaystyle \frac{1}{c}}\left(b_{1}-a_{1}\right)\Biggr\}\right)\left(\mathscr{N}\left(\widetilde{M}\right)\right)^{2}\\
+\max\Biggl\{\left\Vert \overline{N_{4}^{0}}\right\Vert _{1};\left\Vert \overline{N_{4}^{+}}\right\Vert _{1}\Biggr\}.\label{eq:msppp}
\end{multline}

Equations (\ref{eq:ikiioz}), (\ref{eq:sloos})-(\ref{eq:msppp})
and (\ref{eq:poiop}) imply that $\dfrac{\partial\mathcal{T}_{i}\left(\widetilde{M}\right)}{\partial\eta_{j}},\left(j=1,2,3\right),\left(i=1,2,3,4\right)$

are uniformly bounded on their domain when $\widetilde{M}$ varies
within $\mathcal{\mathscr{M}}_{R}$. 

As the $\dfrac{\partial\mathcal{T}_{i}\left(\widetilde{M}\right)}{\partial\eta_{j}}$
are continuous, $\forall\widetilde{M}\in\mathcal{\mathscr{M}}_{R}$,
for $i=1,2,3,4,$ $\mathcal{T}_{i}\left(\widetilde{M}\right)\in C\left(\begin{array}{r}
\mathscr{P}'\end{array};\R\right)$ is differentiable on the domain of $\left(\dfrac{\partial\mathcal{T}_{i}\left(\widetilde{M}\right)}{\partial\eta_{j}}\right)_{i,j}$.
If $\mathcal{T}_{i}\left(\widetilde{M}\right)$ is differentiable
at $\left(\eta_{1},\eta_{2},\eta_{3}\right),$ let $d\left(\mathcal{T}_{i}\left(\widetilde{M}\right)\right)\left(\eta_{1},\eta_{2},\eta_{3}\right)$
denote the differential of $\mathcal{T}_{i}\left(\widetilde{M}\right)$
at $\left(\eta_{1},\eta_{2},\eta_{3}\right).$

$d\left(\mathcal{T}_{i}\left(\widetilde{M}\right)\right)\left(\eta_{1},\eta_{2},\eta_{3}\right)\in\mathscr{L}\left(\R^{3},\R\right),$
space of linear continous functionals on $\R^{3}.$ As 

As $\dfrac{\partial\mathcal{T}_{i}\left(\widetilde{M}\right)}{\partial\eta_{j}},\left(j=1,2,3\right)$
are uniformly bounded, we easily deduce that there exists a constant
$b^{i}$ independant of $\widetilde{M}$ such that 
\begin{equation}
\left\Vert d\left(\mathcal{T}_{i}\left(\widetilde{M}\right)\right)\left(\eta_{1},\eta_{2},\eta_{3}\right)\right\Vert _{\mathscr{L}\left(\R^{3},\R\right)}\leq b^{i}.
\end{equation}

The domain $\mathscr{P}'$ is convex, being a parallelepiped. For
$\left(\eta_{1},\eta_{2},\eta_{3}\right),\left(\eta_{1}',\eta_{2}',\eta_{3}'\right)\in\mathscr{P}'$
if forall $i=1,2,3,4$ $\dfrac{\partial\mathcal{T}_{i}\left(\widetilde{M}\right)}{\partial\eta_{j}}$
are defined on the segment  
\begin{multline*}
   \left[\left(\eta_{1},\eta_{2},\eta_{3}\right),\left(\eta_{1}',\eta_{2}',\eta_{3}'\right)\right]\equiv\\ \left\{ \left(\eta_{1},\eta_{2},\eta_{3}\right)+\alpha\left(\eta_{1}'-\eta_{1},\eta_{2}'-\eta_{2},\eta_{3}'-\eta_{3}\right)\vert\left(0\leq\alpha\leq1\right)\right\},  
\end{multline*}
then by the mean value inequality for all $i=1,2,3,4,$ 
\begin{multline}
\left|\mathcal{T}_{i}\left(\widetilde{M}\right)\left(\eta_{1},\eta_{2},\eta_{3}\right)-\mathcal{T}_{i}\left(\widetilde{M}\right)\left(\eta_{1}',\eta_{2}',\eta_{3}'\right)\right|\\
\leq b^{i}\left\Vert \left(\eta_{1},\eta_{2},\eta_{3}\right)-\left(\eta_{1}',\eta_{2}',\eta_{3}'\right)\right\Vert _{\R^{3}}.
\end{multline}

Then forall $\varepsilon>0,$ with 

\begin{equation}
{\displaystyle \eta_{\varepsilon}=\min_{1\leq i\leq4}\dfrac{\varepsilon}{b^{i}}},\label{eq:qkioalq}
\end{equation}
we have

\begin{multline}
\left\Vert \left(\eta_{1},\eta_{2},\eta_{3}\right)-\left(\eta_{1}',\eta_{2}',\eta_{3}'\right)\right\Vert _{\R^{3}}<\eta_{\varepsilon}\\
\implies\left\Vert \mathcal{T}\left(\widetilde{M}\right)\left(\eta_{1},\eta_{2},\eta_{3}\right)-\mathcal{T}\left(\widetilde{M}\right)\left(\eta_{1}',\eta_{2}',\eta_{3}'\right)\right\Vert _{\R^{4}}\leq\varepsilon.\label{eq:ayueacw}
\end{multline}

The $\dfrac{\partial\mathcal{T}_{i}\left(\widetilde{M}\right)}{\partial\eta_{j}}$
may not be defined only at points of a plane, hence the later points
are adherent to the domain of $\left(\dfrac{\partial\mathcal{T}_{i}\left(\widetilde{M}\right)}{\partial\eta_{j}}\right)_{i,j}.$
Therefore by continuity of $\mathcal{T}\left(\widetilde{M}\right)$
on $\mathscr{P}'$, (\ref{eq:ayueacw}) can be extended to $\mathscr{P}'$:
\begin{multline*}
\forall\varepsilon>0,\exists\eta_{\varepsilon}>0,\forall\widetilde{M}\in\mathscr{M}_{R},\forall\left(\eta_{1},\eta_{2},\eta_{3}\right),\left(\eta_{1}',\eta_{2}',\eta_{3}'\right)\in\mathscr{P}':\\
\left\Vert \left(\eta_{1},\eta_{2},\eta_{3}\right)-\left(\eta_{1}',\eta_{2}',\eta_{3}'\right)\right\Vert <\eta_{\varepsilon}\\
\implies\left\Vert \mathcal{T}\left(\widetilde{M}\right)\left(\eta_{1},\eta_{2},\eta_{3}\right)-\mathcal{T}\left(\widetilde{M}\right)\left(\eta_{1}',\eta_{2}',\eta_{3}'\right)\right\Vert _{\R^{4}}\leq\varepsilon.
\end{multline*}

Therefore $\mathcal{T}\left(\mathscr{M}_{R}\right)$ is equicontinuous
in $C\left(\mathscr{P}';\R^{4}\right)$and by Arzelà-Ascoli theorem,
$\mathcal{T}\left(\mathcal{\mathscr{M}}_{R}\right)$ is relatively
compact in $C\left(\mathscr{P}';\R^{4}\right)$ i.e. $\mathcal{T}$
is compact on $\mathcal{\mathscr{M}}_{R}.$ 
\end{proof}

\subsection{Stable convex set under the operator.}
\begin{prop}
Forall $\widetilde{M}\in\mathscr{M}_{R},$ we have $\mathcal{T}\left(\widetilde{M}\right)\in E^{4}$
and

\begin{multline}
\mathscr{N}\left(\mathcal{T}\left(\widetilde{M}\right)\right)\leq\\4cS\left(1+2\cdot\max\Biggl\{4T;{\displaystyle \frac{2}{c}}\left(b_{1}-a_{1}\right);{\displaystyle \frac{1}{c}}\left(b_{2}-a_{2}\right)\Biggr\}\right)R^{2}\\
+\max\Biggl\{\left\Vert \overline{N_{1}^{0}}\right\Vert _{1};\left\Vert \overline{N_{1}^{-}}\right\Vert _{1};\left\Vert \overline{N_{2}^{0}}\right\Vert _{1};\left\Vert \overline{N_{2}^{-}}\right\Vert _{1};\left\Vert \overline{N_{3}^{0}}\right\Vert _{1};\left\Vert \overline{N_{3}^{+}}\right\Vert _{1};\left\Vert \overline{N_{4}^{0}}\right\Vert _{1};\left\Vert \overline{N_{4}^{+}}\right\Vert _{1}\Biggr\}.\label{eq:olsoep}
\end{multline}
\end{prop}
\begin{proof} We have $\widetilde{M}\in\mathscr{M}_{R}\subset E^{4}\implies\mathcal{T}\left(\widetilde{M}\right)\in E^{4}.$
(proposition (\ref{cor::mzppzo})). Using (\ref{eq:ikiioz}), (\ref{eq:qmpql-1}),
(\ref{eq:qmpql-1-1-1}) and (\ref{eq:aokkao}) we obtain 
\begin{multline}
\left\Vert \dfrac{\partial\mathcal{T}\left(\widetilde{M}\right)}{\partial\eta_{1}}\right\Vert \leq4cS\left(1+2\cdot\max\Biggl\{4T;{\displaystyle \frac{2}{c}}\left(b_{1}-a_{1}\right);{\displaystyle \frac{1}{c}}\left(b_{2}-a_{2}\right)\Biggr\}\right)\\
\left(\mathscr{N}\left(\widetilde{M}\right)\right)^{2}+\max\Biggl\{\left\Vert \overline{N_{2}^{0}}\right\Vert _{1};\left\Vert \overline{N_{2}^{-}}\right\Vert _{1};\left\Vert \overline{N_{3}^{0}}\right\Vert _{1};\left\Vert \overline{N_{3}^{+}}\right\Vert _{1};\left\Vert \overline{N_{4}^{0}}\right\Vert _{1};\left\Vert \overline{N_{4}^{+}}\right\Vert _{1}\Biggr\}.\label{eq:aokkao-1-2-3}
\end{multline}

Similarly we have 

\begin{multline}
\left\Vert \dfrac{\partial\mathcal{T}\left(\widetilde{M}\right)}{\partial\eta_{2}}\right\Vert \leq4cS\left(1+2\cdot\max\Biggl\{2T;{\displaystyle \frac{1}{c}}\left(b_{1}-a_{1}\right);{\displaystyle \frac{1}{c}}\left(b_{2}-a_{2}\right)\Biggr\}\right)\left(\mathscr{N}\left(\widetilde{M}\right)\right)^{2}\\
+\max\Biggl\{\left\Vert \overline{N_{1}^{0}}\right\Vert _{1};\left\Vert \overline{N_{1}^{-}}\right\Vert _{1};\left\Vert \overline{N_{3}^{0}}\right\Vert _{1};\left\Vert \overline{N_{3}^{+}}\right\Vert _{1};\left\Vert \overline{N_{4}^{0}}\right\Vert _{1};\left\Vert \overline{N_{4}^{+}}\right\Vert _{1}\Biggr\}.\label{eq:aokkao-1-2-1-3}
\end{multline}

\begin{multline}
\left\Vert \dfrac{\partial\mathcal{T}\left(\widetilde{M}\right)}{\partial\eta_{3}}\right\Vert \leq4cS\left(1+2\cdot\max\Biggl\{2T;{\displaystyle \frac{1}{c}}\left(b_{1}-a_{1}\right);{\displaystyle \frac{1}{c}}\left(b_{2}-a_{2}\right)\Biggr\}\right)\left(\mathscr{N}\left(\widetilde{M}\right)\right)^{2}\\
+\max\Biggl\{\left\Vert \overline{N_{1}^{0}}\right\Vert _{1};\left\Vert \overline{N_{1}^{-}}\right\Vert _{1};\left\Vert \overline{N_{2}^{0}}\right\Vert _{1};\left\Vert \overline{N_{2}^{-}}\right\Vert _{1};\left\Vert \overline{N_{4}^{0}}\right\Vert _{1};\left\Vert \overline{N_{4}^{+}}\right\Vert _{1}\Biggr\}.\label{eq:aokkao-1-2-1-1-2}
\end{multline}

By taking $\widetilde{N}=\mathcal{T}\left(\widetilde{M}\right)$ and
using (\ref{eq:lospp}), (\ref{eq:aokkao-1-2-3}), (\ref{eq:aokkao-1-2-1-3})
and (\ref{eq:aokkao-1-2-1-1-2}), we obtain for $\widetilde{M}\in\mathscr{M}_{R}$
(\ref{eq:olsoep}).
\end{proof}
Let us set 
\begin{equation}
p\equiv4cS\left(1+2\cdot\max\Biggl\{4T;{\displaystyle \frac{2}{c}}\left(b_{1}-a_{1}\right);{\displaystyle \frac{1}{c}}\left(b_{2}-a_{2}\right)\Biggr\}\right)\label{eq:loaoao}
\end{equation}
and 
\begin{multline}
q''\equiv\\
\max\Biggl\{\left\Vert \overline{N_{1}^{0}}\right\Vert _{1};\left\Vert \overline{N_{1}^{-}}\right\Vert _{1};\left\Vert \overline{N_{2}^{0}}\right\Vert _{1};\left\Vert \overline{N_{2}^{-}}\right\Vert _{1};\left\Vert \overline{N_{3}^{0}}\right\Vert _{1};\left\Vert \overline{N_{3}^{+}}\right\Vert _{1};\left\Vert \overline{N_{4}^{0}}\right\Vert _{1};\left\Vert \overline{N_{4}^{+}}\right\Vert _{1}\Biggr\}.
\end{multline}
Equations (\ref{eq:iiskzi}) yield 
\begin{multline}
\dfrac{\partial\overline{N_{1}^{0}}}{\partial\eta_{2}}\left(\eta_{2},\eta_{3}\right)=-c\dfrac{\partial}{\partial x}N_{1}^{0}\left(x,y\right)+c\dfrac{\partial}{\partial y}N_{1}^{0}\left(x,y\right),\\
\left(x,y\right)=\left(-c\eta_{2}-c\eta_{3},c\eta_{2}-c\eta_{3}\right)\\
\dfrac{\partial\overline{N_{1}^{0}}}{\partial\eta_{3}}\left(\eta_{2},\eta_{3}\right)=-c\dfrac{\partial}{\partial x}N_{1}^{0}\left(x,y\right)-c\dfrac{\partial}{\partial y}N_{1}^{0}\left(x,y\right),\\
\left(x,y\right)=\left(-c\eta_{2}-c\eta_{3},c\eta_{2}-c\eta_{3}\right)\label{aoalal-2-1-2-1-2}
\end{multline}

\begin{multline}
\dfrac{\partial\overline{N_{1}^{-}}}{\partial\eta_{2}}\left(\eta_{2},\eta_{3}\right)=\dfrac{\partial}{\partial t}N_{1}^{-}\left(t,y\right)+c\dfrac{\partial}{\partial y}N_{1}^{-}\left(t,y\right),\\
\left(t,y\right)=\left({\displaystyle \frac{1}{c}}a_{1}+\eta_{2}+\eta_{3},c\eta_{2}-c\eta_{3}\right)\\
\dfrac{\partial\overline{N_{1}^{-}}}{\partial\eta_{3}}\left(\eta_{2},\eta_{3}\right)=\dfrac{\partial}{\partial t}N_{1}^{-}\left(t,y\right)-c\dfrac{\partial}{\partial y}N_{1}^{-}\left(t,y\right),\\
\left(t,y\right)=\left({\displaystyle \frac{1}{c}}a_{1}+\eta_{2}+\eta_{3},c\eta_{2}-c\eta_{3}\right)\label{aoalal-2-1-2-1-2-1}
\end{multline}

\begin{multline}
\left\Vert \overline{N_{1}^{0}}\right\Vert _{\infty}\leq\left\Vert N_{1}^{0}\right\Vert _{1};\left\Vert \dfrac{\partial\overline{N_{1}^{0}}}{\partial\eta_{2}}\right\Vert _{\infty}\leq2c\left\Vert N_{1}^{0}\right\Vert _{1};\left\Vert \dfrac{\partial\overline{N_{1}^{0}}}{\partial\eta_{3}}\right\Vert _{\infty}\leq2c\left\Vert N_{1}^{0}\right\Vert _{1}\\
\left\Vert \overline{N_{1}^{-}}\right\Vert _{\infty}\leq\left\Vert N_{1}^{-}\right\Vert _{1};\left\Vert \dfrac{\partial\overline{N_{1}^{-}}}{\partial\eta_{2}}\right\Vert _{\infty}\leq\left(1+c\right)\left\Vert N_{1}^{-}\right\Vert _{1};\left\Vert \dfrac{\partial\overline{N_{1}^{-}}}{\partial\eta_{3}}\right\Vert _{\infty}\leq\left(1+c\right)\left\Vert N_{1}^{-}\right\Vert _{1}
\end{multline}

\begin{equation}
\left\Vert \overline{N_{1}^{0}}\right\Vert _{1}\leq\max\left\{ 1;2c\right\} \left\Vert N_{1}^{0}\right\Vert _{1};\left\Vert \overline{N_{1}^{-}}\right\Vert _{1}\leq\left(1+c\right)\left\Vert N_{1}^{-}\right\Vert _{1}\label{eq:lsokzkz}
\end{equation}

Similarly, (\ref{eq:ikoik}), (\ref{eq:olooki}), (\ref{eq:lopolop})
yield inequalities which with (\ref{eq:lsokzkz}) yield

\begin{align}
q'' & \leq q\label{eq:loso}
\end{align}
where 
\begin{multline}
q\equiv\\
\max\Biggl\{\max\left\{ 1;2c\right\} \left\Vert N_{1}^{0}\right\Vert _{1};\left(1+c\right)\left\Vert N_{1}^{-}\right\Vert _{1};\max\left\{ 1;2c\right\} \left\Vert N_{2}^{0}\right\Vert _{1};\\\max\left\{ 2;\left(1+c\right)\right\} \left\Vert N_{2}^{-}\right\Vert _{1};
\max\left\{ 1;2c\right\} \left\Vert N_{3}^{0}\right\Vert _{1};\max\left\{ 2;\left(1+c\right)\right\} \left\Vert N_{3}^{+}\right\Vert _{1};\\ \max\left\{ 1;2c\right\} \left\Vert N_{4}^{0}\right\Vert _{1};\left(2+c\right)\left\Vert N_{4}^{+}\right\Vert _{1}\Biggr\};\label{eq:kdipzp}
\end{multline}

Now, (\ref{eq:olsoep}) yields 
\begin{equation}
\mathscr{N}\left(\mathcal{T}\left(\widetilde{M}\right)\right)\leq pR^{2}+q.\label{eq:olol}
\end{equation}

\begin{prop}
\label{prop::eolpa}Suppose 
\begin{equation}
pq\leq\dfrac{1}{4}\label{eq:sjiid}
\end{equation}
and 
\begin{equation}
\dfrac{1-\sqrt{1-4pq}}{2p}\leq R\leq\dfrac{1+\sqrt{1-4pq}}{2p}.\label{eq:asloz}
\end{equation}

Then $\mathcal{T}\left(\mathcal{\mathscr{M}}_{R}\right)\subset\mathcal{\mathcal{\mathscr{M}}_{R}}.$ 
\end{prop}
\begin{proof}
For $\widetilde{M}\in\mathcal{\mathscr{M}}_{R}$ we have (\ref{eq:olol}).
Thus to have $\mathcal{T}\left(\widetilde{M}\right)\in\mathcal{\mathcal{\mathscr{M}}_{R}},\forall\widetilde{M}\in\mathcal{\mathscr{M}}_{R},$ it
is enough that $R>0$ satisfies the inequality $pR^{2}+q\leq R$ i.e.
$pR^{2}-R+q\leq0.$

Now $R>0$ satisfies the preceding inequality if we have: 
\begin{equation}
\begin{cases}
1-4pq\geq0\\
\dfrac{1-\sqrt{1-4pq}}{2p}\leq R\leq\dfrac{1+\sqrt{1-4pq}}{2p}.
\end{cases}\label{eq:oolloz}
\end{equation}
\end{proof}

\subsection{Existence theorem}
\begin{prop}
\label{lsoos}Suppose $pq\leq\dfrac{1}{4}.$ Then the system $\Sigma^{1}$
(\ref{eq:ssdffz}-\ref{eq:yui}) has a non-negative solution $\widetilde{N}=\left(\widetilde{N_{1}},\widetilde{N_{2}},\widetilde{N_{3}},\widetilde{N_{4}}\right)\in C\left(\begin{array}{r}
\mathscr{P}'\end{array};\R^{4}\right)$ such that the partial derivatives $\dfrac{\partial\widetilde{N_{i}}}{\partial\eta_{1}},\dfrac{\partial\widetilde{N_{i}}}{\partial\eta_{2}},\dfrac{\partial\widetilde{N_{i}}}{\partial\eta_{3}}$
are defined in $\mathring{\mathscr{P}'},$ except possibly on a finite
number of planes, are continuous and bounded for $i=1,2,3,4,$ and
such that

\begin{equation}
{\displaystyle \max_{1\leq i\leq4}\left\{ \left\Vert \widetilde{N_{i}}\right\Vert _{\infty},\left\Vert \dfrac{\partial\widetilde{N_{i}}}{\partial\eta_{1}}\right\Vert _{\infty},\left\Vert \dfrac{\partial\widetilde{N_{i}}}{\partial\eta_{2}}\right\Vert _{\infty},\left\Vert \dfrac{\partial\widetilde{N_{i}}}{\partial\eta_{3}}\right\Vert _{\infty}\right\} }\leq\dfrac{1+\sqrt{1-4pq}}{2p}.\label{eq:ppmp}
\end{equation}
\end{prop}
\begin{proof}
For any $\mathcal{\mathscr{M}}_{R}=\left\{ \widetilde{M}=\left(\widetilde{M_{i}}\right)_{i=1}^{4}\in E^{4}:\mathscr{N}\left(\widetilde{M}\right)\leq R\right\} ,\left(R>0\right),$
such that $\left(1-\sqrt{1-4pq}\right)/2p\leq R\leq\left(1+\sqrt{1-4pq}\right)/2p,$
$\mathcal{\mathcal{\mathscr{M}}_{R}}$ is a non-empty convex subset
of $C\left(\mathscr{P}';\R^{4}\right)$ . From propositions \ref{prop::mppp},
\ref{prop::opps} and \ref{prop::eolpa} , $\mathcal{T}$ is continuous
and compact on $\mathcal{\mathcal{\mathscr{M}}_{R}}$, and $\mathcal{T}\left(\mathcal{\mathcal{\mathscr{M}}_{R}}\right)\subset\mathcal{\mathcal{\mathcal{\mathscr{M}}_{R}}}.$
Thus according to Schauder's theorem \ref{sccssq}, $\mathcal{T}$
has a fixed point $\widetilde{N}\in\mathcal{\mathcal{\mathscr{M}}_{R}}.$

We thus have $\mathscr{N}\left(\widetilde{N}\right)\leq R\leq\left(1+\sqrt{1-4pq}\right)/2p$.
From subsection \ref{subsec:Fixed-point-problem}, $\widetilde{N}=\left(\widetilde{N_{1}},\widetilde{N_{2}},\widetilde{N_{3}},\widetilde{N_{4}}\right)$
is a solution of problem $\Sigma^{1}.$ $\widetilde{N}=\left(\widetilde{N_{1}},\widetilde{N_{2}},\widetilde{N_{3}},\widetilde{N_{4}}\right)$
is non-negative from theorem \ref{thm::etaatyeyaiioa-1-1}. On the
other hand, $\widetilde{N}\in\mathscr{M}_{R}\implies\widetilde{N}\in E^{4},$
hence $\widetilde{N}\in C\left(\begin{array}{r}
\mathscr{P}'\end{array};\R^{4}\right)$ and $\dfrac{\partial\widetilde{N_{i}}}{\partial\eta_{1}},\dfrac{\partial\widetilde{N_{i}}}{\partial\eta_{2}},\dfrac{\partial\widetilde{N_{i}}}{\partial\eta_{3}}$
are defined in $\mathring{\mathscr{P}'},$ except possibly on a finite
number of planes, are continuous and bounded $i=1,2,3,4.$
\end{proof}
For $u\in C\left(\left[0;T\right]\times\left[a_{1},b_{1}\right]\times\left[a_{2},b_{2}\right];\R\right)$
let us put 
\begin{equation}
\left\Vert u\right\Vert _{\infty}=\sup_{\left(t,x,y\right)\in\left[0;T\right]\times\left[a_{1},b_{1}\right]\times\left[a_{2},b_{2}\right]}\left|u\left(t,x,y\right)\right|\label{eq:koqq-2-2}
\end{equation}

and for $N=\left(N_{i}\right)_{i=1}^{4}\in C\left(\left[0;T\right]\times\left[a_{1},b_{1}\right]\times\left[a_{2},b_{2}\right];\R^{4}\right)$
\begin{equation}
\left\Vert N\right\Vert =\max_{1\leq i\leq4}\left\Vert N_{i}\right\Vert _{\infty}\label{eq:qkos-2-2}
\end{equation}

For $u:\left[0;T\right]\times\left[a_{1},b_{1}\right]\times\left[a_{2},b_{2}\right]\longrightarrow\R$
with domain $D_{u}\subset\left[0;T\right]\times\left[a_{1},b_{1}\right]\times\left[a_{2},b_{2}\right]$
such that $u$ is bounded on $D_{u},$ let us put 
\begin{equation}
\left\Vert u\right\Vert _{\infty}=\sup_{\left(t,x,y\right)\in D_{u}}\left|u\left(t,x,y\right)\right|\label{eq:koqq-2-1-1}
\end{equation}

and for $N=\left(N_{i}\right)_{i=1}^{4}\in C\left(D_{N};\R^{4}\right),$
\begin{equation}
\left\Vert N\right\Vert =\max_{1\leq i\leq4}\left\Vert N_{i}\right\Vert _{\infty}.\label{eq:qkos-2-1-2}
\end{equation}

\subsection{Proof of the main theorem\textit{ \ref{thm:Suppose-.-Then}}}
\begin{proof}
From proposition (\ref{lsoos}), and the change of variables $\mathscr{F}$
, we deduce that the problem $\Sigma^{0}$ has a solution $N=\left(N_{i}\right)_{i=1}^{4}\equiv N\left(t,x,y\right)$
such that $N\left(t,x,y\right)=\widetilde{N}\left(\eta_{1},\eta_{2},\eta_{3}\right)$
where $\widetilde{N}$ is a solution of $\Sigma^{1}.$

We have 
\[
\left\Vert \widetilde{N_{i}}\right\Vert _{\infty}=\sup_{\left(\eta_{1},\eta_{2},\eta_{3}\right)\in\mathscr{P}'}\left|\widetilde{N_{i}}\left(\eta_{1},\eta_{2},\eta_{3}\right)\right|=\sup_{\left(t,x,y\right)\in\mathscr{P}}\left|N_{i}\left(t,x,y\right)\right|=\left\Vert N_{i}\right\Vert _{\infty}
\]
hence (\ref{eq:ppmp}) yields 
\begin{align}
{\displaystyle \left\Vert N\right\Vert =\max_{1\leq i\leq4}\left\Vert N_{i}\right\Vert _{\infty}} & \leq\dfrac{1+\sqrt{1-4pq}}{2p}.
\end{align}

One has from the inverse of the change of variables,$$N\left(t,x,y\right)=\widetilde{N}\left(\dfrac{1}{c}x;\dfrac{1}{2}t-\dfrac{1}{2c}x+\dfrac{1}{2c}y;\dfrac{1}{2}t-\dfrac{1}{2c}x-\dfrac{1}{2c}y\right)$$
hence 
\begin{equation}
\begin{cases}
\dfrac{\partial N_{i}}{\partial t}=\dfrac{1}{2}\dfrac{\partial\widetilde{N_{i}}}{\partial\eta_{2}}+\dfrac{1}{2}\dfrac{\partial\widetilde{N_{i}}}{\partial\eta_{3}}\\
\dfrac{\partial N_{i}}{\partial x}=\dfrac{1}{c}\dfrac{\partial\widetilde{N_{i}}}{\partial\eta_{1}}-\dfrac{1}{2c}\dfrac{\partial\widetilde{N_{i}}}{\partial\eta_{2}}-\dfrac{1}{2c}\dfrac{\partial\widetilde{N_{i}}}{\partial\eta_{3}}\\
\dfrac{\partial N_{i}}{\partial y}=\dfrac{1}{2c}\dfrac{\partial\widetilde{N_{i}}}{\partial\eta_{2}}-\dfrac{1}{2c}\dfrac{\partial\widetilde{N_{i}}}{\partial\eta_{3}}
\end{cases}.\label{eq:osoo}
\end{equation}

We deduce that 
\begin{align}
\left\Vert \dfrac{\partial N_{i}}{\partial t}\right\Vert _{\infty} & \leq\dfrac{1}{2}\left\Vert \dfrac{\partial\widetilde{N_{i}}}{\partial\eta_{2}}\right\Vert _{\infty}+\dfrac{1}{2}\left\Vert \dfrac{\partial\widetilde{N_{i}}}{\partial\eta_{3}}\right\Vert _{\infty}\nonumber \\
\left\Vert \dfrac{\partial N_{i}}{\partial t}\right\Vert _{\infty} & \leq\dfrac{1+\sqrt{1-4pq}}{2p}\label{eq:slso}
\end{align}
\begin{align}
\left\Vert \dfrac{\partial N_{i}}{\partial x}\right\Vert _{\infty} & \leq\dfrac{1}{c}\left\Vert \dfrac{\partial\widetilde{N_{i}}}{\partial\eta_{1}}\right\Vert _{\infty}+\dfrac{1}{2c}\left\Vert \dfrac{\partial\widetilde{N_{i}}}{\partial\eta_{2}}\right\Vert _{\infty}+\dfrac{1}{2c}\left\Vert \dfrac{\partial\widetilde{N_{i}}}{\partial\eta_{3}}\right\Vert _{\infty}\nonumber \\
\left\Vert \dfrac{\partial N_{i}}{\partial x}\right\Vert _{\infty} & \leq\dfrac{2}{c}\dfrac{1+\sqrt{1-4pq}}{2p}\label{eq:slops}
\end{align}
\begin{align}
\left\Vert \dfrac{\partial N_{i}}{\partial y}\right\Vert _{\infty} & \leq\dfrac{1}{2c}\left\Vert \dfrac{\partial\widetilde{N_{i}}}{\partial\eta_{2}}\right\Vert _{\infty}+\dfrac{1}{2c}\left\Vert \dfrac{\partial\widetilde{N_{i}}}{\partial\eta_{3}}\right\Vert _{\infty}\nonumber \\
\left\Vert \dfrac{\partial N_{i}}{\partial y}\right\Vert _{\infty} & \leq\dfrac{1}{c}\dfrac{1+\sqrt{1-4pq}}{2p}.\label{eq:slsp}
\end{align}
Thus we can deduce (\ref{eq:loosqz}).

Also, $\dfrac{\partial\widetilde{N_{i}}}{\partial\eta_{1}},\dfrac{\partial\widetilde{N_{i}}}{\partial\eta_{2}},\dfrac{\partial\widetilde{N_{i}}}{\partial\eta_{3}}$
are defined except possibly on a finite number of planes including
the four planes with respective equations 
\begin{equation}
-c\eta_{2}-c\eta_{3}=a_{1},-c\eta_{1}-2c\eta_{3}=a_{2},c\eta_{1}+2c\eta_{2}=b_{2},2c\eta_{1}+c\eta_{2}+c\eta_{3}=b_{1}.\label{eq:lxoo}
\end{equation}
We deduce that $\dfrac{\partial N_{i}}{\partial t},\dfrac{\partial N_{i}}{\partial x},\dfrac{\partial N_{i}}{\partial y}$
are defined except possibly on a finite number of planes including
the transformed of the planes with equations (\ref{eq:lxoo}) by the
inverse $\mathscr{F}^{-1}$; direct calculations using $\mathscr{F}^{-1}$
give the equations (\ref{eq:ldoopz}).

Suppose that the problem $\Sigma^{0}$ have two solutions $M$ and
$N$ satisfying 
\begin{equation}
\begin{cases}
{\displaystyle \left\Vert M\right\Vert =\max_{1\leq i\leq4}\left\Vert M_{i}\right\Vert _{\infty}}\leq\dfrac{1+\sqrt{1-4pq}}{2p}\\
{\displaystyle \left\Vert N\right\Vert =\max_{1\leq i\leq4}\left\Vert N_{i}\right\Vert _{\infty}}\leq\dfrac{1+\sqrt{1-4pq}}{2p}
\end{cases}.\label{mmspp}
\end{equation}
$\widetilde{M},\widetilde{N}$ defined on $\mathscr{P}'$ by $\widetilde{M}\left(\eta_{1},\eta_{2},\eta_{3}\right)=M\left(t,x,y\right)$
and $\widetilde{N}\left(\eta_{1},\eta_{2},\eta_{3}\right)=N\left(t,x,y\right)$
are solutions of problem $\Sigma^{1}.$

We have $\left\Vert \widetilde{N_{i}}\right\Vert _{\infty}=\left\Vert N_{i}\right\Vert _{\infty}\text{ and }\left\Vert \widetilde{M_{i}}\right\Vert _{\infty}=\left\Vert M_{i}\right\Vert _{\infty},$
hence \\
$\left\Vert \widetilde{N}\right\Vert =\max_{1\leq i\leq4}\left\Vert \widetilde{N_{i}}\right\Vert _{\infty}=\max_{1\leq i\leq4}\left\Vert N_{i}\right\Vert _{\infty}=\left\Vert N\right\Vert $
and $\left\Vert \widetilde{M}\right\Vert =\left\Vert M\right\Vert .$

Equations (\ref{mmspp}) yields ${\displaystyle \left\Vert \widetilde{M}\right\Vert }\leq\left(1+\sqrt{1-4pq}\right)/2p,{\displaystyle \left\Vert \widetilde{N}\right\Vert }\leq\left(1+\sqrt{1-4pq}\right)/2p.$
Hence (\ref{eq:lsoopa-2}) yields 
\begin{multline}
\left\Vert \mathcal{T}\left(\widetilde{M}\right)-\mathcal{T}\left(\widetilde{N}\right)\right\Vert \leq p'\cdot2\cdot\dfrac{1+\sqrt{1-4pq}}{2p}\left\Vert \widetilde{M}-\widetilde{N}\right\Vert \\
\leq\dfrac{p'}{p}\cdot\left(1+\sqrt{1-4pq}\right)\left\Vert \widetilde{M}-\widetilde{N}\right\Vert .\label{eq:lsoopa-1}
\end{multline}
But $\widetilde{M}$ and $\widetilde{N}$ are fixed points of $\mathcal{T},$
hence (\ref{eq:lsoopa-1}) is written 
\begin{align}
\left\Vert \widetilde{M}-\widetilde{N}\right\Vert  & \leq\dfrac{p'}{p}\cdot\left(1+\sqrt{1-4pq}\right)\left\Vert \widetilde{M}-\widetilde{N}\right\Vert \label{eq:lsoopa-1-1}
\end{align}
i.e. 
\begin{align}
\left(1-\dfrac{p'}{p}\cdot\left(1+\sqrt{1-4pq}\right)\right)\left\Vert \widetilde{M}-\widetilde{N}\right\Vert  & \leq0.\label{eq:lsoopa-1-1-1}
\end{align}

From (\ref{eq:lsoopa-2}) and (\ref{eq:loaoao}) we have 
\begin{equation}
\dfrac{p'}{p}=\dfrac{\max\Biggl\{ T,{\displaystyle \frac{1}{c}}\left(b_{1}-a_{1}\right),{\displaystyle \frac{1}{c}}\left(b_{2}-a_{2}\right)\Biggl\}}{1+2\cdot\max\Biggl\{4T;{\displaystyle \frac{2}{c}}\left(b_{1}-a_{1}\right);{\displaystyle \frac{1}{c}}\left(b_{2}-a_{2}\right)\Biggr\}}\equiv\dfrac{m_{1}}{1+2m_{2}}
\end{equation}
 $m_{1}\leq m_{2};$$\dfrac{p'}{p}=\dfrac{m_{1}}{1+2m_{2}}=\dfrac{1}{\dfrac{1}{m_{2}}+2}\cdot\dfrac{m_{1}}{m_{2}};$
and successively $\begin{cases}
\dfrac{1}{\dfrac{1}{m_{2}}+2}<\dfrac{1}{2}\\
\dfrac{m_{1}}{m_{2}}\leq1
\end{cases};$

$\dfrac{p'}{p}=\dfrac{1}{\dfrac{1}{m_{2}}+2}\cdot\dfrac{m_{1}}{m_{2}}<\dfrac{1}{2}.$

On the other hand $1+\sqrt{1-4pq}<2,$ hence $\dfrac{p'}{p}\cdot\left(1+\sqrt{1-4pq}\right)<1$;
i.e. 
\begin{equation}
1-\dfrac{p'}{p}\cdot\left(1+\sqrt{1-4pq}\right)>0.\label{eq:loodkk}
\end{equation}

From (\ref{eq:lsoopa-1-1-1}) and (\ref{eq:loodkk}) we deduce that
\begin{align}
\left\Vert \widetilde{M}-\widetilde{N}\right\Vert  & \leq0.\label{eq:lsoopa-1-1-1-1-1}
\end{align}
Therefore $\widetilde{M}=\widetilde{N}$ and $M=N.$ Hence the uniqueness.
\end{proof}

\section{Conclusion}

We show that under some condition on the data, the initial-boundary
value problem in a rectangle for the two dimension Broadwell's 4-velocity
model has a continuous unique non-negative solution bounded with its
first partial derivatives. We provide a bound for the solution and
the derivatives. Our perspectives are now to study the case where
the data are general. 

\rule[0.5ex]{1\columnwidth}{1pt}

\end{document}